\documentclass[11pt]{article}
\usepackage{amsmath,amssymb,amsthm,graphicx}
\usepackage[mathscr]{eucal}
\usepackage[cm]{fullpage}
\usepackage{xcolor}
\usepackage[english]{babel}
\usepackage[latin1]{inputenc}

\usepackage{tikz,lipsum,lmodern}
\usepackage[most]{tcolorbox}

\newtheorem{theorem}{Theorem}[section]
\newtheorem{proposition}[theorem]{Proposition}

\newtheorem{lemma}[theorem]{Lemma}

\def\ker{\mathop{\mathrm{Ker}}\nolimits}
\def\im{\mathop{\mathrm{Im}}\nolimits}
\def\rank{\mathop{\mathrm{rank}}\nolimits}
\def\dom{\mathop{\mathrm{Dom}}\nolimits}
\def\fix{\mathop{\mathrm{Fix}}\nolimits}
\def\endo{\mathop{\mathrm{End}}\nolimits}
\def\auto{\mathop{\mathrm{Aut}}\nolimits}
\def\inn{\mathop{\mathrm{Inn}}\nolimits}
\def\nor{\mathrm{N}}
\def\N{\mathbb N}

\def\id{\mathrm{id}}
\def\O{\mathcal{O}}
\def\PT{\mathcal{PT}}
\def\T{\mathcal{T}}
\def\Sym{\mathcal{S}}
\def\I{\mathcal{I}}
\def\C{\mathcal{C}}
\def\D{\mathcal{D}}
\def\PO{\mathcal{PO}}
\def\POI{\mathcal{POI}}
\def\OD{\mathcal{OD}}
\def\PODI{\mathcal{PODI}}
\def\POD{\mathcal{POD}}
\def\OP{\mathcal{OP}}
\def\POPI{\mathcal{POPI}}
\def\POP{\mathcal{POP}}
\def\OR{\mathcal{OR}}
\def\PORI{\mathcal{PORI}}
\def\POR{\mathcal{POR}}
\def\J{\mathscr{J}}

\def\L{\mathscr{L}}
\def\H{\mathscr{H}} 
\def\varwidetilde#1{\widetilde{#1\,}} 
\def\ov#1{\overline{#1}} 

\newcommand{\lastpage}{\addresss}

\newcommand{\addresss}{\small \sf  

\noindent{\sc De Biao Li}, 
School of Mathematics and Statistics, 
Lanzhou University, 
Lanzhou 730000, 
Gansu, 
P. R. China.\\
e-mail: lidb19@lzu.edu.cn

\medskip

\noindent{\sc V\'\i tor H. Fernandes}, 
Center for Mathematics and Applications (CMA), 
FCT NOVA and Department of Mathematics, FCT NOVA, 
Faculdade de Ci\^encias e Tecnologia, 
Universidade Nova de Lisboa, 
Monte da Caparica, 
2829-516 Caparica, 
Portugal. 
e-mail: vhf@fct.unl.pt
}

\author{De Biao Li and V\'\i tor H. Fernandes\footnote{This work is funded by national funds through the FCT - Funda\c c\~ao para a Ci\^encia e a Tecnologia, 
I.P., under the scope of the projects UIDB/00297/2020 and UIDP/00297/2020 (Center for Mathematics and Applications).}
}

\title{Endomorphisms of semigroups of oriented transformations}

\begin{document}

\maketitle

\vspace*{-1cm}

\begin{abstract}
In this paper, we characterize the monoid of endomorphisms of the semigroup of all oriented full transformations of a finite chain, as well as the monoid of endomorphisms of the semigroup of all oriented partial transformations 
and the monoid of endomorphisms of the semigroup of all oriented partial permutations of a finite chain. 
Characterizations of the monoids of endomorphisms of the subsemigroups of all orientation-preserving transformations of the three semigroups aforementioned are also given. 
In addition, we compute the number of endomorphisms of each of these six semigroups. 
\end{abstract}

\medskip



\section*{Introduction} 

Let $\Omega$ be a set. We denote by $\PT(\Omega)$ the monoid (under composition) of all 
partial transformations on $\Omega$, by $\T(\Omega)$ the submonoid of $\PT(\Omega)$ of all 
full transformations on $\Omega$, by $\I(\Omega)$ 
the \textit{symmetric inverse monoid} on $\Omega$, i.e. 
the inverse submonoid of $\PT(\Omega)$ of all 
partial permutations on $\Omega$, 
and by $\Sym(\Omega)$ the \textit{symmetric group} on $\Omega$, 
i.e. the subgroup of $\PT(\Omega)$ of all 
permutations on $\Omega$. 
For $n\in\N$, let $\Omega_n$ be a set with $n$ elements. 
As usual, we denote $\PT(\Omega_n)$, $\T(\Omega_n)$, $\I(\Omega_n)$ and $\Sym(\Omega_n)$ 
simply by $\PT_n$, $\T_n$, $\I_n$ and $\Sym_n$, respectively. 

\smallskip 

Next, suppose that $\Omega_n$ is a chain, e.g. $\Omega_n=\{1<2<\cdots<n\}$. 

A partial transformation $\alpha\in\PT_n$ is called \textit{order-preserving}  
[respectively, \textit{order-reversing}] 
if $x\leqslant y$ implies $x\alpha\leqslant y\alpha$, 
[respectively, $x\alpha\geqslant y\alpha$] 
for all $x,y \in \dom(\alpha)$.  
An order-preserving or order-reversing transformation is also called \textit{monotone}. 
We denote by $\PO_n$ [respectively, $\POD_n$] 
the submonoid of $\PT_n$ of all order-preserving [respectively, monotone] 
partial transformations, by $\O_n$ [respectively, $\OD_n$] the submonoid of $\T_n$ of all order-preserving [respectively, monotone]  
full transformations and by $\POI_n$ [respectively, $\PODI_n$]  the submonoid of $\I_n$ of all order-preserving [respectively, monotone] 
partial permutations. 

Let $s=(a_1,a_2,\ldots,a_t)$
be a sequence of $t$ ($t\geqslant0$) elements
from the chain $\Omega_n$. 
We say that $s$ is \textit{cyclic} 
[respectively, \textit{anti-cyclic}] if there
exists no more than one index $i\in\{1,\ldots,t\}$ such that
$a_i>a_{i+1}$ [respectively, $a_i<a_{i+1}$],
where $a_{t+1}$ denotes $a_1$.
Notice that, the sequence $s$ is cyclic
[respectively, anti-cyclic] if and only if $s$ is empty or there exists
$i\in\{0,1,\ldots,t-1\}$ such that 
$a_{i+1}\leqslant a_{i+2}\leqslant \cdots\leqslant a_t\leqslant a_1\leqslant \cdots\leqslant a_i $ 
[respectively, $a_{i+1}\geqslant a_{i+2}\geqslant \cdots\geqslant a_t\geqslant a_1\geqslant \cdots\geqslant a_i $] (the index
$i\in\{0,1,\ldots,t-1\}$ is unique unless $s$ is constant and
$t\geqslant2$). We also say that $s$ is \textit{oriented} if $s$ is cyclic or $s$ is anti-cyclic. 
See \cite{Catarino&Higgins:1999,Higgins&Vernitski:2022,McAlister:1998}. 

Given a partial transformation $\alpha\in\PT_n$ such that
$\dom(\alpha)=\{a_1<\cdots<a_t\}$, with $t\geqslant0$, we 
say that $\alpha$ is \textit{orientation-preserving} 
[respectively, \textit{orientation-reversing}, \textit{oriented}] if the sequence of its images
$(a_1\alpha,\ldots,a_t\alpha)$ is cyclic [respectively, anti-cyclic, oriented].  
We denote by $\POP_n$ the submonoid of $\PT_n$ 
of all orientation-preserving
transformations and by $\POR_n$ the
submonoid of $\PT_n$ of all
oriented  transformations. 
It is easy to show 
that the product of two orientation-preserving or of two
orientation-reversing transformations is orientation-preserving and
the product of an orientation-preserving transformation by an
orientation-reversing transformation, or vice-versa, is 
orientation-reversing.

Notice that $\PO_n\subseteq\POD_n\subseteq\POR_n$ and 
$\PO_n\subseteq\POP_n \subseteq\POR_n$, by definition.

Denote by $\OP_n$ the submonoid of $\PT_n$ 
of all orientation-preserving full 
transformations and by $\OR_n$ the
submonoid of $\PT_n$ of all
oriented full transformations, 
i.e. $\OP_n=\POP_n\cap\T_n$ and $\OR_n=\POR_n\cap\T_n$. 
Consider also the submonoids $\POPI_n=\POP_n\cap\I_n$ and $\PORI_n=\POR_n\cap\I_n$ of $\PT_n$.

Notice that, by definition, we have $\O_n\subseteq\OD_n\subseteq\OR_n$,  
$\O_n\subseteq\OP_n \subseteq\OR_n$, 
$\POI_n\subseteq\PODI_n\subseteq\PORI_n$ and  
$\POI_n\subseteq\POPI_n \subseteq\PORI_n$.

\smallskip 

Let us consider the following permutations of $\Omega_n$ of order $n$ and $2$, respectively: 
$$
g=\begin{pmatrix} 
1&2&\cdots&n-1&n\\
2&3&\cdots&n&1
\end{pmatrix} 
\quad\text{and}\quad  
h=\begin{pmatrix} 
1&2&\cdots&n-1&n\\
n&n-1&\cdots&2&1
\end{pmatrix}. 
$$
It is clear that $g$ preserves the orientation and $h$ reverses the order, whence $g$ and $h$ are both oriented permutations.  
Moreover, for $n\geqslant3$, $g$ together with $h$ generate the well-known \textit{dihedral group} $\D_{2n}$ of order $2n$ 
(considered as a subgroup of $\Sym_n$). In fact, we have 
$$
\D_{2n}=\langle g,h\mid g^n=h^2=1, gh=hg^{n-1}\rangle=\{1,g,g^2,\ldots,g^{n-1}, h,hg,hg^2,\ldots,hg^{n-1}\}. 
$$
Denote also by $\C_n$ the \textit{cyclic group} of order $n$ generated by $g$, i.e. $\C_n=\{1,g,g^2,\ldots,g^{n-1}\}$.  

\smallskip 

The Hasse diagram in Figure \ref{diagram}, with respect to the inclusion relation and
where {\bf1} denotes the trivial monoid, clarifies the
relationship between these various semigroups. 
\begin{figure}[ht]
\begin{center}
\begin{picture}(130,150)(0,35)
\unitlength=1.15pt
{
\put(117.5,107.5){$\bullet$} \put(122.5,107.5){\scriptsize $\T_n$}
\put(67.5,157.5){$\bullet$} \put(75,157.5){\scriptsize $\PT_n$}
\put(17.5,107.5){$\bullet$} \put(7,107.5){\scriptsize $\I_n$}
\put(120,110){\line(-1,1){50}} \put(120,90){\line(0,1){20}}
\put(70,140){\line(0,1){20}} \put(20,110){\line(1,1){50}}
\put(20,90){\line(0,1){20}}
\put(67.5,57.5){$\bullet$} 
\put(70,60){\line(1,1){50}}
\put(70,60){\line(-1,1){50}} \put(70,60){\line(0,-1){20}}
\put(72.5,57.5){\scriptsize $\mathcal{S}_n$}
\put(67.5,47.5){$\bullet$} 
\put(73.0,47.5){\scriptsize \boldmath$\mathcal{D}_{2n}$}
\put(70.0,50){\line(5,4){51}}
\put(70.0,50){\line(-5,4){51}}
}
\put(67.5,37.5){$\bullet$} \put(70,40){\line(1,1){30}}
\put(70,40){\line(-1,1){30}} \put(74,37.5){\scriptsize \boldmath$\mathcal{C}_n$}
\put(37.5,67.5){$\bullet$} \put(40,70){\line(1,1){30}}
\put(40,70){\line(-1,1){20}} 
\put(45,69){\scriptsize \boldmath$\mathcal{POPI}_n$}
\put(97.5,67.5){$\bullet$} \put(100,70){\line(-1,1){30}}
\put(100,70){\line(1,1){20}} 
\put(103,67.5){\scriptsize \boldmath$\mathcal{OP}_n$}
\put(17.5,87.5){$\bullet$} \put(20,90){\line(1,1){50}}
\put(-12,87.5){\scriptsize \boldmath$\mathcal{PORI}_n$}
\put(117.5,87.5){$\bullet$} \put(120,90){\line(-1,1){50}}
\put(123,87.5){\scriptsize \boldmath$\mathcal{OR}_n$}
\put(67.5,97.5){$\bullet$} \put(70,100){\line(0,1){40}}
\put(73.5,98){\scriptsize \boldmath$\mathcal{POP}_n$}
\put(68,138){$\bullet$} 
\put(57,142.7){\scriptsize \boldmath$\mathcal{POR}_n$}
{
\put(78,18.5){$\bullet$} \put(80,20){\line(1,1){30}}
\put(80,20){\line(-1,1){30}} \put(84,17){\scriptsize {\bf1}}
{
\put(80,20){\line(-1,2){10}} }
\put(47.5,47.5){$\bullet$} \put(50,50){\line(1,1){30}}
\put(50,50){\line(-1,1){20}} \put(25,47.5){\scriptsize
$\mathcal{POI}_n$}
{
\put(50,50){\line(-1,2){10}} }
\put(107.5,47.5){$\bullet$} \put(110,50){\line(-1,1){30}}
\put(110,50){\line(1,1){20}} \put(113.5,47.5){\scriptsize
$\mathcal{O}_n$}
{
\put(110,50){\line(-1,2){10}} }
\put(27.5,67.5){$\bullet$} \put(30,70){\line(1,1){50}}
\put(-2,67.5){\scriptsize $\mathcal{PODI}_n$}
{
\put(30,70){\line(-1,2){10}} }
\put(127.5,67.5){$\bullet$} \put(130,70){\line(-1,1){50}}
\put(132.5,67.5){\scriptsize $\mathcal{OD}_n$}
{
\put(130,70){\line(-1,2){10}} }
\put(77,77){$\bullet$} \put(80,80){\line(0,1){40}}
\put(60,78.5){\scriptsize $\mathcal{PO}_n$}
{
\put(80,80){\line(-1,2){10}} }
\put(77.5,117.5){$\bullet$} \put(54,118.5){\scriptsize
$\mathcal{POD}_n$}
{
\put(80,120){\line(-1,2){10}} } }
\end{picture}
\end{center}
\caption{Hasse diagram with respect to the inclusion relation}
\label{diagram}
\end{figure}

\smallskip 

Observe that the inclusion relation represented in Figure \ref{diagram} may not be strict for some values of $n$. 
For instance, $\OR_2=\OD_2=\T_2$,  $\PORI_2=\PODI_2=\I_2$ and $\POR_2=\POD_2=\PT_2$. 

\smallskip 

The notion of an orientation-preserving
transformation on a finite chain was introduced by McAlister in \cite{McAlister:1998} and,
independently, by Catarino and Higgins in \cite{Catarino&Higgins:1999}. 
Several properties of the monoid $\OP_n$ were investigated in these two papers. 
Since then, semigroups of orientated transformations have been studied in many other papers. 
See, for example, 
\cite{Araujo&etal:2011,
Arthur&Ruskuc:2000,
Catarino:1998,
Delgado&Fernandes:2000, 
Dimitrova&al:2012, 
East&al:2018,
Fernandes:2000,
Fernandes&Gomes&Jesus:2009,  
Fernandes&Gomes&Jesus:2011,
Fernandes&al:2016,
Fernandes&Quinteiro:2011,
Higgins&Vernitski:2022,
Zhao&Fernandes:2015}. 

\smallskip 

Characterizing endomorphisms of transformation semigroups is a classical problem. 
For instance, this was done by Schein and Teclezghi \cite{Schein&Teclezghi:1997,Schein&Teclezghi:1998,Schein&Teclezghi:1998b} 
for $\I_n$ in 1997 and for $\T_n$ and $\PT_n$ in 1998,  
and by Mazorchuk \cite{Mazorchuk:2002} for the Brauer-type semigroups in 2002. 
More recently, 
Fernandes et al. \cite{Fernandes&al:2010} found a description of the endomorphisms of $\O_n$,  
Fernandes and Santos \cite{Fernandes&Santos:2019} determined the endomorphisms of $\POI_n$ and $\PO_n$ and 
Fernandes and Li \cite{Fernandes&Li:2022sub} gave a characterization of the endomorphisms of $\OD_n$, $\PODI_n$ and $\POD_n$. 
Descriptions of automorphisms of semigroups of order-preserving transformations and of some of their extensions, such as semigroups of monotone transformations or semigroups of oriented transformations, can be found in \cite{Araujo&etal:2011}. 

In this paper, we complete the picture by giving descriptions of the monoids of endomorphisms of the remain semigroups of the diagram of Figure \ref{diagram}, 
namely of $\OP_n$, $\POPI_n$, $\POP_n$, $\OR_n$, $\PORI_n$ and $\POR_n$. 
Moreover, we also determine the number of endomorphisms of each of these six semigroups. 

\medskip 

Throughout this paper, whenever not mentioned explicitly, we consider $n\geqslant3$. 

\section{Preliminaries} 

Let $S$ be a semigroup. For completion, we recall the definition of
the Green equivalence relations $\mathscr{R}$, $\mathscr{L}$, $\mathscr{H}$ and
$\mathscr{J}$: for all $s, t\in S$,
\begin{description}
\item $s\mathscr{R} t$ if and only if  $sS^1=tS^1$, 
\item $s\mathscr{L} t$ if and only if $S^1s=S^1t$, 
\item $s\mathscr{H} t$ if and only if $s\mathscr{L} t$ and $s\mathscr{R} t$, 
\item $s\mathscr{J} t$ if and only if $S^1sS^1=S^1tS^1$
\end{description}
(as usual, $S^1$ denotes $S$ with identity adjoined \textit{if necessary}). 
Associated to Green's relation $\mathscr{J}$ there is a quasi-order
$\leqslant_ \mathscr{J}$ on $S$ defined by
$$
s\leqslant_ \mathscr{J} t \text{ if and only if }  S^1sS^1\subseteq S^1tS^1,
$$
for all $s, t\in S$. Notice that, for every $s, t\in S$, we have
$s\,\mathscr{J}\,t$ if and only if $s\leqslant_ \mathscr{J}t$ and
$t\leqslant_ \mathscr{J}s$. Denote by $J_{s}$ the $\mathscr{J}$-class
of the element $s\in S$. As usual, a partial order relation
$\leqslant_ \mathscr{J}$ is defined on the set $S/\mathscr {J}$ by
setting $J_{s}\leqslant_ \mathscr{J}J_{t}$ if and only if
$s\leqslant_ \mathscr{J}t$, for all $s, t\in S$. For $s, t\in S$, we
write $s<_\mathscr{J}t$ and also $J_{s}<_\mathscr{J}J_{t}$ if
and only if $s\leqslant_ \mathscr{J}t$ and $(s, t)\not\in\mathscr{J}$. 
Recall that any endomorphism of semigroups preserves Green's relations and the quasi-order $\leqslant_ \mathscr{J}$. 

\smallskip 

Given a semigroup $S$ and a subset $X$ of $S$, we denote by $E(X)$ the set of idempotents of $S$ that belong to $X$.
An \textit{ideal} of $S$ is a subset $I$ of $S$ such that
$S^1IS^1\subseteq I$. By convenience, we admit the empty set as an
ideal. A \textit{Rees congruence} of $S$ is a congruence associated
to an ideal of $S$: if $I$ is an ideal of $S$, the Rees congruence
$\rho_I$ is defined by $(s,t)\in\rho_I$ if and only if $s=t$ or
$s,t\in I$, for all $s, t\in S$. 
The group of automorphisms of $S$ and the monoid of endomorphisms of $S$ are denoted by 
$\auto(S)$ and $\endo(S)$, respectively. 

\smallskip 

For general background on semigroups and standard notations, we refer the reader to Howie's book \cite{Howie:1995}.

\medskip 

Let $S\in\{\O_n,\POI_n,\PO_n,\OP_n,\POPI_n,\POP_n,\OR_n,\PORI_n,\POR_n,\T_n,\I_n,\PT_n\}$. 
We have the following descriptions of Green's relations in the semigroup $S$: 
\begin{description}
\item $s\mathscr{L} t$ if and only if $\im(s)=\im(t)$, 
\item $s\mathscr{R} t$ if and only if $\ker(s)=\ker(t)$, 
\item $s\mathscr{J} t$ if and only if $|\im(s)|=|\im(t)|$, and 
\item $s\mathscr{H} t$ if and only if $\ker(s)=\ker(t)$ and  $\im(s)=\im(t)$,
\end{description} 
for all $s,t\in S$. If $S\in\{\POI_n,\POPI_n,\PORI_n,\I_n\}$, for Green's relation $\mathscr{R}$, we have, even more simply, 
\begin{description}
\item
$s\mathscr{R} t$ if and only if $\dom(s)=\dom(t)$, 
\end{description} 
for all $s,t\in S$.

\smallskip 

Recall that
$\OP_n=\langle \O_n\cup\{g\}\rangle$, 
$\POPI_n=\langle\POI_n\cup\{g\}\rangle$, 
$\POP_n=\langle \PO_n\cup\{g\}\rangle$, 
$\OR_n=\langle \OP_n\cup\{h\}\rangle = \langle \O_n\cup\{g,h\}\rangle$, 
$\PORI_n=\langle\POPI_n\cup\{h\}\rangle=\langle \POI_n\cup\{g,h\}\rangle$ 
and $\POR_n=\langle \POP_n\cup\{h\}\rangle=\langle \PO_n\cup\{g,h\}\rangle$.  
More specifically, if $T\in\{\O_n, \POI_n, \PO_n\}$ and $x\in\langle T\cup\{g\}\rangle$ [respectively, $x\in\langle T\cup\{g,h\}\rangle$] then 
there exist $t\in T$ and $u\in\C_n=\langle g \rangle$ [respectively, $u\in\D_{2n}=\langle g,h \rangle$] such that $x=ut$. 

Moreover, for $T\in\{\O_n, \POI_n, \PO_n\}$, we have that $T$, $\langle T\cup\{g\}\rangle$ and $\langle T\cup\{g,h\}\rangle$
are regular monoids and $E(\langle T\cup\{g\}\rangle)=E(\langle T\cup\{g,h\}\rangle)$.  Furthermore, 
$\POI_n$, $\POPI_n$ and $\PORI_n$ are inverse monoids and 
$E(\POI_n)=E(\POPI_n)=E(\PORI_n)=E(\I_n)=\{\id|_X\mid X\subseteq\Omega_n\}$. 

Recall also that, for the partial order $\leqslant_ \mathscr{J}$ on $S$, the 
quotient $S/\mathscr{J}$ is a chain with $n$ elements
for $S\in\{\O_n,\OP_n,\OR_n,\T_n\}$ and with $n+1$ elements for the remain cases of $S$.
More precisely, 
$$
S/\mathscr{J}=\{J_0^S <_\mathscr{J} J_1^S <_\mathscr{J} \cdots
<_\mathscr{J} J_n^S \}
$$ 
when $S\in\{\POI_n,\PO_n,\POPI_n,\PORI_n,\POP_n,\POR_n,\I_n,\PT_n\}$, and
$$
S/\mathscr{J}=\{J_1^S <_\mathscr{J} \cdots <_\mathscr{J} J_n^S\}
$$ 
when $S\in\{\O_n,\OP_n,\OR_n,\T_n\}$. Here 
$$
J_k^S=\{s\in S\mid |\im(s)|=k\},
$$ 
with $k$ suitably defined. 
For $0\leqslant  k\leqslant  n$, let 
$I_k^S=\{s\in S\mid |\im(s)|\leqslant k\}$. 
Clearly $I_k^S$ is an ideal of $S$. Since $S/\mathscr{J}$ is a
chain, it follows that
$$
\{I_k^S\mid 0\leqslant  k\leqslant  n\}
$$
is the set of all ideals of $S$ (see \cite{Fernandes:2001}). 

Let $\alpha\in\PT_n$. Recall that the \textit{rank} of $\alpha$, denoted by $\rank(\alpha)$, is the size of $\im(\alpha)$. 

Observe also that $\O_n$, $\POI_n$ and $\PO_n$ are aperiodic monoids (i.e. with only trivial subgroups); the
group $\mathscr{H}$-classes of rank $k$ of 
$\OP_n$, $\POPI_n$ and $\POP_n$ are cyclic groups of order $k$, for $1\leqslant k\leqslant n$; and the
group $\mathscr{H}$-classes of rank $k$ of 
$\OR_n$, $\PORI_n$ and $\POR_n$
are dihedral groups of order $2k$, for $3\leqslant k\leqslant n$, and cyclic groups of order $k$, for $k=1,2$.
In particular, $\C_n$ is the group of units of $\OP_n$, $\POPI_n$ and $\POP_n$ and, 
on the other hand, $\D_{2n}$ is the group of units of $\OR_n$, $\PORI_n$ and $\POR_n$. 

See \cite{Fernandes:2000,Fernandes:2001,Fernandes&Gomes&Jesus:2009,Ganyushkin&Mazorchuk:2009,Gomes&Howie:1992}. 

\smallskip 

Let $S$ be a finite semigroup and let $J$ be a $\mathscr{J}$-class
of $S$. Denote by $B(J)$ the set of all elements $s\in S$ such
that $J\not\leqslant_ {\mathscr{J}}J_s$. It is clear that $B(J)$ is an ideal of $S$.
We associate to $J$ a relation
$\pi_J$ on $S$ defined by: for all $s,t\in S$, we have $s\,\pi_J\, t$ if
and only if
\begin{description}
  \item (a)
  $s=t$; or
  \item (b)
  $s,t\in B(J)$; or
  \item (c) 
  $s,t\in J$ and $s\,\mathscr{H}\,t$.
\end{description}
Then the relation $\pi_J$ is a congruence on $S$ (see \cite{Fernandes:2000}). 

Assume that $J$ is regular and take a group $\mathscr{H}$-class
$H_0$ of $J$. Also, suppose that there exists a mapping
$$\begin{array}{lllll}
\varepsilon&:&J&\longrightarrow&H_0\\
&&s&\longmapsto&\tilde{s}
\end{array}
$$
which satisfies the
following property: given $s,t\in J$ such that $st\in J$,
there exist $x,y\in H_0$ such that
\begin{eqnarray*}
\mbox{$b\,\mathscr{H}\,t$ implies $\varwidetilde{sb}=x\tilde{s}\tilde{b}$}\\
\mbox{$a\,\mathscr{H}\,s$ implies
$\varwidetilde{at}=\tilde{a}\tilde{t}y$}.
\end{eqnarray*}

To each congruence $\pi$ on $H_0$, we associate a relation
$\rho^J_\pi$ on $S$ defined by: given $s,t\in S$, we have 
$$
\mbox{$s\,\rho^J_\pi\,t$
if and only if $s\,\pi_J\, t$ and $s,t\in J$ implies
$\tilde{s}\,\pi\,\tilde{t}$.}
$$
Therefore, the relation $\rho^J_\pi$ is also a congruence on $S$ (see \cite{Fernandes&Gomes&Jesus:2009}). 
Moreover, Fernandes et al. proved in \cite{Fernandes&Gomes&Jesus:2009}:

\begin{theorem} [{\cite[Theorem 3.3]{Fernandes&Gomes&Jesus:2009}}] \label{con} 
The congruences of $S\in\{\OP_n,\POPI_n,\POP_n,\OR_n,\PORI_n,\POR_n\}$ are exactly the
congruences $\rho^{J^S_k}_\pi$, where $\pi$ is a congruence on a fixed maximal subgroup 
contained in $J^S_k$, for
$k\in\{1,2,\ldots,n\}$, and the universal congruence.
\end{theorem}

Notice that  $\rho^{J^S_1}_\pi$ is the identity congruence on $S$.

\section{On the endomorphisms}\label{endo} 

In this section we aim to present a description of the monoid of endomorphisms for each of the semigroups 
$\OP_n$, $\OR_n$, $\POPI_n$, $\PORI_n$, $\POP_n$ and $\POR_n$. 

\medskip 

The following result generalizes Lemma 2.1 of \cite{Fernandes&Santos:2019}. 

\begin{lemma}\label{lr}
Let $S$ be a regular semigroup. Let $I$ be a subset of $S$ and $\phi$ be an endomorphism of $S$ such that $I$ is a kernal class of $\phi$ and for all $s,t\in S\setminus I$, $s\phi=t\phi$ implies that $s\mathscr{H}t$. Let $s,t\in S\setminus I$. Then:
\begin{enumerate}
\item $s\mathscr{L}t$ if and only if $s\phi\mathscr{L}t\phi$; 
\item $s\mathscr{R}t$ if and only if $s\phi\mathscr{R}t\phi$. 
\end{enumerate}
\end{lemma} 
\begin{proof}
We prove the lemma for Green's relation $\mathscr{L}$. The proof for $\mathscr{R}$ is similar. 

Let $s,t\in S\setminus I$. 
If $s\mathscr{L}t$, then $s\phi\mathscr{L}t\phi$, since any homomorphism of semigroups preserves Green's relations. Conversely, suppose that $s\phi\mathscr{L}t\phi$. As $S$ is regular, then $S\phi$ is also regular, whence $s\phi$ and $t\phi$ are also $\L$-related in $S\phi$ and so, 
for some $u,v\in S$, we have $s\phi=(u\phi)(t\phi)$ and $t\phi=(v\phi)(s\phi)$. 
Hence $s\phi=(ut)\phi$ and so $ut\in S\setminus I$, 
since $s\in S\setminus I$ and $I$ is a kernel class of $\phi$. 
Thus $s\mathscr{H}ut$, by hyphothesis. 
Analogously, $t\mathscr{H}vs$ and so there exist $x,y\in S$ such that $s=xut$ and $t=yvs$. Therefore $s\mathscr{L}t$, as required. 
\end{proof}

Recall the following lemmas from \cite{Fernandes&al:2010} and \cite{Fernandes&Santos:2019}:

\begin{lemma}[{\cite[Lemma 2.4]{Fernandes&al:2010}}] \label{o}
Let $\lfloor \frac{n+2}{2} \rfloor\leqslant k\leqslant n-1$. Then there exist idempotents $e_1,\ldots, e_{n-1}\in J_k^{\O_n}$ such that $e_i e_j\in I_{k-1}^{\O_n}$ or $e_j e_i\in I_{k-1}^{\O_n}$ for all $1\leqslant i< j\leqslant n-1$.
\end{lemma}

\begin{lemma}[{\cite[Lemma 2.3]{Fernandes&Santos:2019}}] \label{poi}
Let $1\leqslant k\leqslant n-1$. Then there exist idempotents $e_1,\ldots, e_{n}\in J_k^{\POI_n}$ such that $e_i e_j\in I_{k-1}^{\POI_n}$ for all $1\leqslant i< j\leqslant n$.
\end{lemma}

Since $\O_{n}\subseteq\OP_{n}\subseteq\OR_n$ and $\POI_{n}\subseteq\POPI_{n}\subseteq\PORI_n,\POP_n\subseteq\POR_n$, it immediately follows: 
  
\begin{lemma} \label{all} 
Let $S\in\{\OP_n,\OR_n,\POPI_n,\PORI_n,\POP_n,\POR_n\}$ and let $\lfloor \frac{n+2}{2} \rfloor\leqslant k\leqslant n-1$. 
Then there exist idempotents $e_1,\ldots, e_{n-1}\in J_k^S$ such that $e_i e_j\in I_{k-1}^S$ or $e_j e_i\in I_{k-1}^S$ for all $1\leqslant i< j\leqslant n-1$.
\end{lemma}

\medskip 

In what follows, for $S\in\{\OP_n,\POPI_n,\POP_n,\OR_n,\PORI_n,\POR_n\}$, 
$k\in\{1,\ldots,n\}$ and a congruence $\pi$ on a maximal subgroup $H_k$ contained in $J^S_k$, 
we denote the congruence 
$\rho^{J^S_k}_\pi$ of  $S$ simply by $\rho^k_\pi$, provided there is no danger of ambiguity. 

The proof of the following lemma uses a reasoning analogous to what can be found in \cite{Fernandes&al:2010,Fernandes&Santos:2019}, 
which we include, on one hand, for the sake of completeness and, on the other hand, because there exist some critical own specificities.   

\begin{lemma}\label{2leqn-1}
Let $S\in\{\OP_n,\POPI_n,\POP_n,\OR_n,\PORI_n,\POR_n\}$, $1\leqslant k\leqslant n-1$ and 
$\phi$ an endomorphism of $S$ such that $\ker(\phi)=\rho^k_\pi$, for some congruence $\pi$ on a fixed maximal subgroup $H_k$ contained in $J^S_k$. 
Then:
\begin{enumerate}
\item if $S\in\{\OP_n,\OR_n\}$ then $k=1$;
\item if $k\geqslant2$ then $S\in\{\POPI_n,\POP_n,\PORI_n,\POR_n\}$, $k=n-1$, 
$J_n^S\phi=J^S_n$, $J_{n-1}^{S}\phi\subseteq J_{1}^{S}$ and $I_{n-2}^{S}\phi=\{\emptyset\}$.
\end{enumerate}
\end{lemma} 
\begin{proof} 
Observe that it suffices to show the second property. So, suppose that $k\geqslant2$. 

\smallskip 

During this proof we omit the superscript $^S$ in the notation of $\mathscr{J}$-classes and ideals.

\smallskip 

Let $f$ be the idempotent of $S$ such that $I_{k-1}\phi=\{f\}$. Then, we also have $f\phi^{-1}=I_{k-1}$. 

Since a homomorphism of semigroups preserves Green's relations, there exists $m\in\{1,2,\ldots,n\}$ such that $J_{k}\phi\subseteq J_{m}$.  
Moreover, since a homomorphism of semigroups preserves the quasi-order $\leqslant_ \J$,  
we have $\rank(f)\leqslant m$ and 
$$ 
(\bigcup_{i=k}^{n}J_i)\phi\subseteq \bigcup_{i=m}^{n}J_i 
$$
and so 
$$\displaystyle 
(\bigcup_{i=k}^{n}E(J_i))\phi = E(\bigcup_{i=k}^{n}J_i)\phi \subseteq E(\bigcup_{i=m}^{n}J_i) = \bigcup_{i=m}^{n}E(J_i)
$$
From the injectivity of $\phi$ in $\bigcup_{i=k}^{n}E(J_i)$, we also get $m\leqslant k$. 

\smallskip 

Next, we aim to show that $\rank(f)<m<k$. 

We first suppose that $\rank(f)=k$. Then $m=k$, i.e. $J_{k}\phi\subseteq J_{k}$.
Since $\phi$ is injective on $E(J_{k})$, we have $E(J_{k})\phi=E(J_{k})$.
Then, there exists $e\in E(J_{k})$ such that $e\phi=f$, which contradicts the equality $f\phi^{-1}=I_{k-1}$. 
Hence $\rank(f)\leqslant k-1$ and so $f\phi=f$.

Secondly, suppose that $m=k$. Then, as above, we get $E(J_{k})\phi=E(J_{k})$. 
Clearly, there exists $e\in E(J_{k})$ such that $\fix(f)\nsubseteq \fix(e)$. Then $fe\neq f$ and so, being $t\in e\phi^{-1}$, we have 
$f\neq fe=f\phi t\phi=(ft)\phi=f$, which is a contradiction. Hence $m<k$.

Finally, suppose that $\rank(f)=m$. Let $e\in E(J_{k})$. Then $\rank(f)=m=\rank(e\phi)$. 
Since $e\phi \cdot f= e\phi f\phi=(ef)\phi=f=(fe)\phi=f\phi e\phi=f\cdot e\phi$, we deduce that $\ker(f)=\ker(e\phi)$ and $\im(f)=\im(e\phi)$. 
Then $f\mathscr{H} e\phi$ and so, as $f$ and $e\phi$ are idempotents, we get $f=e\phi$, which contradicts 
$f\phi^{-1}=I_{k-1}$. Hence, we also have $\rank (f)<m$.

\smallskip 

Now, by Lemma \ref{lr}, the number of $\mathscr{L}$-classes in $J_{k}\phi$ is equal to the number of $\mathscr{L}$-classes in $J_{k}$, 
which is $\binom{n}{k}$.
Since the number of $\mathscr{L}$-classes in $J_{m}$ is $\binom{n}{m}$ and $J_{k}\phi\subseteq J_{m}$, 
it follows that $\binom{n}{k}\leqslant \binom{n}{m}$ and so, as $m<k$, we must have $k -1\geqslant \lfloor \frac{n}{2} \rfloor$, i.e. 
$k\geqslant \lfloor \frac{n+2}{2} \rfloor$. 
Then, by Lemma \ref{all}, there exist idempotents $e_1,\ldots, e_{n-1}\in J_{k}$ such that 
$e_i e_j\in I_{k -1}$ or $e_j e_i\in I_{k -1}$ for all $1\leqslant i< j\leqslant n-1$. 
Since $f\cdot s\phi=f\phi\cdot s\phi=(fs)\phi=f$ for all $s\in S$, 
we have $\im(f)\subseteq \im(s\phi)$. Hence $\im(f)\subseteq \im(e_{i}\phi)$ for all $1\leqslant i\leqslant n-1$. 
For all $1\leqslant i< j\leqslant n-1$, we have $e_{i}\phi\cdot e_{j}\phi=(e_{i}e_{j})\phi=f$ or $e_{j}\phi\cdot e_{i}\phi=(e_{j}e_{i})\phi=f$, 
whence $\im(e_{i}\phi)\cap \im(e_{j}\phi)\subseteq \im(f)$, as $e_{i}\phi$ and $e_{j}\phi$ are idempotents, and so $\im(e_{i}\phi)\cap \im(e_{j}\phi)=\im(f)$.  

Let $E_i=\im(e_{i}\phi)\setminus \im(f)$ for $1\leqslant i\leqslant n-1$. Since $\rank(e_{i}\phi)=m>\rank(f)$ for $1\leqslant i\leqslant n-1$, the sets $E_1,\ldots,E_{n-1}$ are nonempty and $E_i\cap E_j=\emptyset$ for $1\leqslant i< j\leqslant n-1$.  
Moreover, we get $\im(f)\cap \bigcup_{i=1}^{n-1}E_{i}=\emptyset$ 
and $|\im(f)|+\sum_{i=1}^{n-1}|E_{i}|=|\im(f)|+|\bigcup_{i=1}^{n-1}E_{i}|=|\im(f)\cup \bigcup_{i=1}^{n-1}E_{i}|\leqslant n$. 
Since $|E_{1}|=|E_{2}|=\cdots=|E_{n-1}|$, we must have $|E_{1}|=|E_{2}|=\cdots=|E_{n-1}|=1$ and $|\im(f)|\leqslant 1$. 

Suppose that $|\im(f)|=1$. Then $m=2$, i.e. $J_{k}\phi\subseteq J_{2}$, and since $f\cdot s\phi=f\phi\cdot s\phi=(fs)\phi=f$, 
whence  $\im(f)\subsetneq \im(s\phi)$, 
for all $s\in J_{k}$, it follows that $J_{k}\phi$ has at most $\binom{n-1}{1}=n-1$ distinct $\mathscr{L}$-classes. 
Hence, by Lemma \ref{lr}, $J_{k}$ also has at most $n-1$ distinct $\mathscr{L}$-classes, and so $\binom{n}{k}\leqslant n-1$, 
which contradicts $2\leqslant k\leqslant n-1$.

Therefore, we conclude that $S\in\{\POPI_n,\POP_n,\PORI_n,\POR_n\}$, $f=\emptyset$ and $m=1$, i.e. $J_{k}\phi\subseteq J_{1}$. 
Since $J_1$ has exactly $n$ distinct $\mathscr{L}$-classes, by applying again Lemma \ref{lr}, 
we deduce that  $J_{k}$ has at most $n$ distinct  $\mathscr{L}$-classes, 
whence $\binom{n}{k}\leqslant n$, and so $k=n-1$,  
since $2\leqslant k \leqslant n-1$. Thus $J_{n-1}\phi\subseteq J_{1}$ and $I_{n-2}\phi=\{\emptyset\}$.

\smallskip 

It remains to show that $J_n\phi=J_n$. Since $\phi$ is injective in $J_n$, which is the group of units of $S$, we have $J_n\phi\simeq J_n$.  
Now, recall that the maximal subgroups of $J_k$ are: cyclic of order $k$, for $S\in\{\POPI_n,\POP_n\}$ and $1\leqslant k\leqslant n$; 
cyclic of order $k$, for $S\in\{\PORI_n,\POR_n\}$ and $k=1,2$; and dihedral of order $2k$, for $S\in\{\PORI_n,\POR_n\}$ and $3\leqslant k\leqslant n$. 
Therefore, we must have $J_n\phi\subseteq J_n$ and so $J_n\phi=J_n$, as required. 
\end{proof}

\smallskip 

Next, notice that 
$$
g^k=\begin{pmatrix} 
1&2&\cdots&n-k&n-k+1&\cdots&n\\
1+k&2+k&\cdots&n&1&\cdots&k
\end{pmatrix}, 
\quad\text{i.e.}\quad  
ig^k=\left\{\begin{array}{ll}
i+k & \mbox{if $1\leqslant i\leqslant n-k$}\\
i+k-n & \mbox{if $n-k+1\leqslant i\leqslant n$,} 
\end{array}\right.
$$
and 
$$
hg^k=\begin{pmatrix} 
1&\cdots&k&k+1&\cdots&n\\
k&\cdots&1&n&\cdots&k+1
\end{pmatrix}, 
\quad\text{i.e.}\quad  
ihg^k=\left\{\begin{array}{ll}
k-i+1 & \mbox{if $1\leqslant i\leqslant k$}\\
n+k-i+1 & \mbox{if $k+1\leqslant i\leqslant n$,}  
\end{array}\right.
$$
for $0\leqslant k\leqslant n-1$. 
Observe that $g^ih=hg^{n-i}$ and $hg^i=g^{n-i}h$ for $0\leqslant i\leqslant n-1$.  
Recall that $\D_{2n}=\langle g,h\rangle$ admits as proper normal subgroups:
\begin{enumerate}
  \item $\langle g^2,h\rangle$, 
  $\langle g^2,hg\rangle$ and $\langle g^p\rangle$, with $p$ a divisor of $n$, when $n$ is even;
  \item $\langle g^p\rangle$, with $p$ a divisor of $n$, when $n$ is odd.
\end{enumerate}
See \cite{Dummit&Foote} for more details.
Recall also  the elementary fact that all subgroups of $\C_n$ are normal and cyclic. 

\smallskip 

Now, for endomorphisms with kernel congruence associated to $J_n^S$, we have: 

\begin{lemma}\label{rhon}
Let $S\in\{\OP_n,\POPI_n,\POP_n,\OR_n,\PORI_n,\POR_n\}$ and 
$\phi$ an endomorphism of $S$ such that $\ker(\phi)=\rho^n_\pi$, for some congruence $\pi$ on the group of units $J_n^S$ of $S$. 
Then $\phi$ is an endomorphism satisfying one of the following properties: 
\begin{enumerate}
\item there exist idempotents $e,f\in S$ with $e\not=f$ and $ef=fe=f$ such
that $J_n\phi=\{e\}$ and $I_{n-1}^S\phi=\{f\}$;

\item  if $S\in\{\POPI_n,\POP_n\}$, 
there exists a group-element $g_0$ in $S$ of order $p$, with $p$ a divisor of $n$ and $p>1$, 
such that $g\phi=g_0$ and $I_{n-1}^S\phi=\{\emptyset\}$;

\item if $S\in\{\PORI_n,\POR_n\}$, 
there exist group-elements $h_0$ of order $2$ and $g_0$ of order $p$, with $p$ a divisor of $n$ and $p>1$, in some maximal subgroup of $S$ such that 
$\langle g_0,h_0\rangle$ is a dihedral group of order $2p$, $g\phi=g_0$, $h\phi=h_0$ and $I_{n-1}^S\phi=\{\emptyset\}$; 

\item  if $S\in\{\OR_n,\PORI_n,\POR_n\}$, 
there exist a group-element $h_0$ in $S$ of order $2$ and an idempotent $f$ of $S$ with rank less than or equal to $2$ such that 
$h_0f=fh_0=f$, $I_{n-1}^S\phi=\{f\}$ and:
\begin{enumerate}
\item $g\phi=h_0^2$ and $h\phi=h_0$; or 
\item $g\phi=h_0$ and $h\phi=h_0^2$, with $n$ even; or 
\item $g\phi=h\phi=h_0$, with $n$ even. 
\end{enumerate}
\end{enumerate}
\end{lemma} 
\begin{proof} During this proof we omit the superscript $^S$ in the notation of $\mathscr{J}$-classes and ideals.

\smallskip 

Let $H=J_n\phi$ and let $f$ be the idempotent of $S$ such that $I_{n-1}\phi=\{f\}$. 
Then, we also have $f\phi^{-1}=I_{n-1}$, whence $f\in I_{n-1}$ and so $Hf=fH=\{f\}$.  

Denote by $\varphi$ the homomorphism of groups $\phi|_{J_n}:J_n\longrightarrow H$ 
and consider $\ker(\varphi)$ as a normal subgroup of $J_n$. 
Then $H$ is isomorphic to $J_n/_{\ker(\varphi)}$. 
Observe that the mapping $J_n/_{\ker(\varphi)}\longrightarrow H$ 
such that $s\!\ker(\varphi)\longmapsto s\phi$, for $s\in J_n$,  is a group isomorphism. 

\smallskip 

Let us suppose that $S\in\{\OP_n,\POPI_n,\POP_n\}$. Then $J_n=\C_n=\langle g\rangle$, $H=\langle g\phi\rangle$ and
$g\phi\cdot f=f\cdot g\phi=f$. In particular, $\fix(f)\subseteq\fix(g\phi)$. 

Obviously, we have $\ker(\varphi)=\langle g^p\rangle$, for some divisor $p$ of $n$. 
Hence, as $|\langle g^p\rangle|=\frac{n}{p}$, we get $|H|=|\C_n/_{\langle g^p\rangle}|=p$ and that $g\phi$ is a group-element in $S$ of order $p$. 

If $p=1$ then $g\phi$ is an idempotent of $S$ and so $\phi$ is an endomorphism of $S$ satisfying 1. 

If $p>1$ then $\fix(g\phi)=\emptyset$ and so we must have $S\in\{\POPI_n,\POP_n\}$ and $f=\emptyset$. 
Thus, in this case, $\phi$ is an endomorphism of $S$ satisfying 2. 

\smallskip 

Now, suppose that $S\in\{\OR_n,\PORI_n,\POR_n\}$. 
Then $J_n=\D_{2n}=\langle g,h\rangle$, $H=\langle g\phi,h\phi\rangle$ and
$sf=fs=f$, for all $s\in H$. In particular, $\fix(f)\subseteq\fix(s)$, for all $s\in H$. 
Notice that, if $s\in H$ is a non-idempotent element then $|\fix(s)|\leqslant 2$ 
and, in addition, if $s$ is an orientation-preserving transformation then $\fix(s)=\emptyset$. 

Next, we consider each of the possibilities for $\ker(\varphi)$. 

\smallskip 

First, we consider that $\ker(\varphi)=D_{2n}$. 

Clearly, in this case, $g\phi=h\phi$ is an idempotent of $S$ and so $\phi$ is an endomorphism of $S$ satisfying 1. 

\smallskip 

Secondly, admit that $\ker(\varphi)=\langle g^p\rangle$, with $p$ a divisor of $n$. 

Then, as $|\langle g^p\rangle|=\frac{n}{p}$, we have $|H|=|\D_{2n}/_{\langle g^p\rangle}|=2p$ and that $g\phi$ and $h\phi$ are group-elements of $S$  of orders $p$ and $2$, respectively.  

If $p=1$ then, clearly, $g\phi$ is an idempotent of $S$, $(h\phi)^2=g\phi$ and so $\phi$ is an endomorphism of $S$ satisfying 4(a). 

Let $p>1$. Then $H$ has at least $3$ non-idempotent elements which can not be all non-orientation-preserving transformations, 
since the product of two distinct orientation-reversing transformations of a maximal subgroup of $S$ is a non-idempotent orientation-preserving transformation. 
Hence, there exists a non-idempotent orientation-preserving transformation $s$ in $H$ and so $\fix(f)\subseteq\fix(s)=\emptyset$. 
Therefore $S\in\{\POPI_n,\POP_n\}$ and $f=\emptyset$. 
Furthermore, $H$ must a dihedral group of order $2p$ and so $\phi$ is an endomorphism of $S$ satisfying 3. 

\smallskip 

Next, let us suppose that $\ker(\varphi)=\langle g^2,h\rangle$, with $n$ even. 

Since $\langle g^2,h\rangle=\{g^{2k}, hg^{2k}\mid 0\leqslant k\leqslant \frac{n-2}{2}\}$, 
we have $|H|=|\D_{2n}/_{\langle g^2,h\rangle}|=2$ and that $g\phi$ and $h\phi$ are group-elements of $S$ of orders $2$ and $1$, respectively.  
It follows that  $\phi$ is an endomorphism of $S$ satisfying 4(b). 

\smallskip 

Finally, we consider the case $\ker(\varphi)=\langle g^2,hg\rangle$, with $n$ even. 

Since $\langle g^2,hg\rangle=\{g^{2k}, hg^{2k+1}\mid 0\leqslant k\leqslant \frac{n-2}{2}\}$, 
we have $|H|=|\D_{2n}/_{\langle g^2,h\rangle}|=2$ and that $g\phi$ and $h\phi$ are both group-elements of $S$ of order $2$.  
Thus  $\phi$ is an endomorphism of $S$ satisfying 4(c), as required. 
\end{proof}

Notice that, under the conditions of 4 of the preceding lemma, 
since $\fix(f)\subseteq\fix(h_0)$, if $f\neq\emptyset$ 
then $h_0$ is not an orientation-preserving transformation. 

\medskip 

Let $T$ be a subsemigroup of $\PT_n$ and consider $\inn(T)=\{\phi _{T}^{\sigma}\mid \sigma\in\Sym_n~\text{and}~\sigma^{-1}T\sigma\subseteq T\}$, 
where $\phi _{T}^{\sigma}$ denotes the transformation of $T$ defined by $t\phi _{T}^{\sigma}=\sigma^{-1}t\sigma$, for all $t\in T$, for each $\sigma\in \Sym_n$. 
Observe that, for $\sigma\in \Sym_n$, if $\sigma^{-1}T\sigma\subseteq T$ then $\sigma^{-1}T\sigma=T=\sigma T\sigma^{-1}$. Moreover, 
it is clear that $\inn(T)$ is a subgroup of the automorphism group $\auto(T)$ of $T$. We call \textit{inner automorphisms}  of $T$ (under $\Sym_n$) to the elements of $\inn(T)$. 

\smallskip 

Recall that $T$ is said to be $\Sym_n$-\textit{normal} if $\sigma^{-1}T\sigma\subseteq T$, for all $\sigma\in\Sym_n$. 
In 1975, Sullivan \cite[Theorem 2]{Sullivan:1975} proved that $\inn(T)=\auto(T)\simeq\Sym_n$, for any $\Sym_n$-normal subsemigroup $T$ of $\PT_n$ containing a constant mapping. 
Moreover, we obtain an isomorphism $\Phi:\Sym_n\longrightarrow\auto(T)$ by defining $\sigma\Phi=\phi _{T}^{\sigma}$, for all $\sigma\in \Sym_n$. 

\medskip 

Let $S\in\{\I_n,\PT_n\}$ and take $I_1=I_1^S$. Notice that $I_1=I_1^{\POI_n}=I_1^{\POPI_n}=I_1^{\PORI_n}$, if $S=\I_n$, and  
$I_1=I_1^{\PO_n}=I_1^{\POP_n}=I_1^{\POR_n}$, if $S=\PT_n$. Moreover, it is clear that $I_1$ is an $\Sym_n$-normal subsemigroup of $\PT_n$ containing a constant mapping. 
Therefore, by Sullivan's Theorem, we have 
\begin{equation}\label{I1}
\auto(I_1)=\inn(I_1)=\{\phi _{I_1}^{\sigma}\mid \sigma\in\Sym_n\}\simeq\Sym_n.
\end{equation}

Now, let $G$ be any subgroup of $\Sym_n$ and take ${\I^G_1}=I_1\cup G$. Then ${\I^G_1}$ is a submonoid of $\PT_n$ admiting $G$ as group of units and such that 
${\I^G_1}/\mathscr{J}=\{\{\emptyset\} <_\mathscr{J} J_1^S <_\mathscr{J} G\}$. 
Since ${\I^G_1}$ contains constant idempotent transformations with arbitrary images, by \cite[Lemma 2.2]{Araujo&etal:2011}, we have 
$$
\auto({\I^G_1})=\inn({\I^G_1}).
$$

Denote the \textit{normalizer} of $G$ in $\Sym_n$ by
$$
\nor_{\Sym_n}(G)=\{\sigma\in \Sym_n \mid\sigma^{-1}G\sigma=G\}.
$$
Then, we have the following result.

\begin{proposition}
Let $G$ be a subgroup of $\Sym_n$. Then 
$\nor_{\Sym_n}(G)\simeq \inn({\I^G_1})=\auto({\I^G_1})$.
\end{proposition} 
\begin{proof}
Let $\sigma\in\nor_{\Sym_n}(G)$. Then, since clearly $\sigma^{-1}I_1\sigma\subseteq I_1$ and, by definition, 
$\sigma^{-1}G\sigma\subseteq G$, we have $\sigma^{-1}{\I^G_1}\sigma\subseteq {\I^G_1}$. Hence, we may define a mapping 
$\varphi:\nor_{\Sym_n}(G)\longrightarrow \inn(\I_1^G)$ by $\sigma\varphi =\phi_{\I^G_1}^{\sigma}$ for all $\sigma \in \nor_{\Sym_n}(G)$. 
It is easy to show that $\varphi$ is a homomorphism. 
Moreover, by restricting $\sigma\varphi$ to $I_1$ for each $\sigma\in\nor_{\Sym_n}(G)$, 
we get the restriction to $\nor_{\Sym_n}(G)$ of the above defined isomorphism 
 $\Phi:\Sym_n\longrightarrow\auto(I_1)$, from which follows that $\varphi$ is injective. 
On the other hand, let $\phi\in\inn(\I_1^G)$. 
Then $\phi=\phi_{\I^G_1}^{\sigma}$, for some $\sigma\in\Sym_n$ such that $\sigma^{-1}{\I^G_1}\sigma\subseteq {\I^G_1}$. 
As $\phi$ is an automorphism of ${\I^G_1}$, it must preserve the quasi-order $\leqslant_ \mathscr{J}$, 
whence $G\phi\subseteq G$ ($J_1^S\phi\subseteq J_1^S$ and $\emptyset\phi=\emptyset$) 
and so $\sigma^{-1}G\sigma\subseteq G$. Thus $\sigma\in \nor_{\Sym_n}(G)$, as required.
\end{proof}

\medskip 

Next, we aim to construct an endomorphism of $\PT_n$. 

Observe that $J_{n-1}^{\PT_n}=J_{n-1}^{\I_n}\cup J_{n-1}^{\T_n}$ (a disjoint union). 

Let us define a mapping $\phi_{\scriptscriptstyle 0}:\PT_n\longrightarrow\PT_n$ by: 
\begin{enumerate}
\item $\phi_{\scriptscriptstyle 0}|_{J_n^{\PT_n}}=\id|_{J_n^{\PT_n}}$; 

\item For $s\in J_{n-1}^{\I_n}$ such that $\dom(s)=\Omega_n\setminus\{i\}$ and 
$\im(s)=\Omega_n\setminus\{j\}$, for some $1\leqslant i,j\leqslant n$, 
let $s\phi_{\scriptscriptstyle 0}=\binom{i}{j}$; 

\item For $s\in J_{n-1}^{\T_n}$ such that $\im(s)=\Omega_n\setminus\{k\}$ and $\{i,j\}$ is the only non-singleton kernel class of $s$ (i.e. $i\neq j$ and $is=js$), 
for some $1\leqslant i,j,k\leqslant n$, 
let 
$s\phi_{\scriptscriptstyle 0}=\left(\begin{array}{cc}
i&j\\
k&k
\end{array}\right)$; 

\item $I_{n-2}^{\PT_n}\phi_{\scriptscriptstyle 0}=\{\emptyset\}$.
\end{enumerate}

Clearly, $\phi_{\scriptscriptstyle 0}$ is a well defined mapping. Moreover, it is a routine matter to prove the following lemma. 

\begin{lemma} \label{phi1} 
Let $S$ be a subsemigroup of $\PT_n$ such that $S\phi_{\scriptscriptstyle 0}\subseteq S$. 
Then, the mapping $\phi_{\scriptscriptstyle 0}$ is an endomorphism of $\PT_n$ and 
the restriction of $\phi_{\scriptscriptstyle 0}$ to $S$ may also be seen as an endomorphism of $S$.  
In particular, for any $S\in\{\POI_n,\PO_n,\POPI_n,\POP_n,\PORI_n,\POR_n\}$, the mapping $\phi_{\scriptscriptstyle 0}|_S:S\longrightarrow S$ is an endomorphism of $S$. 
\end{lemma} 

Notice that, for $S\in\{\POPI_n,\POP_n,\PORI_n,\POR_n\}$, we have $\ker(\phi_{\scriptscriptstyle 0}|_S)=\rho_{\pi}^{n-1}$, where $\pi$ is the universal congruence on a maximal subgroup of $S$ contained in $J_{n-1}^S$. Furthermore,
$J_n^S\phi_{\scriptscriptstyle 0}=J^S_n$, $J_{n-1}^{S}\phi_{\scriptscriptstyle 0}\subseteq J_{1}^{S}$ and $I_{n-2}^{S}\phi_{\scriptscriptstyle 0}=\{\emptyset\}$.

\medskip 

Let $G=J_n^S$ and $\sigma\in\nor_{\Sym_n}(G)$. Then, it is clear that $\phi_\sigma^S=\phi_{\scriptscriptstyle 0}|_S\phi_{{\I^G_1}}^{\sigma}$ is an endomorphism of $S$ 
such that  $J_n^S\phi_\sigma^S=J^S_n$, $J_{n-1}^{S}\phi_\sigma^S\subseteq J_{1}^{S}$ and $I_{n-2}^{S}\phi^S_\sigma=\{\emptyset\}$.
Moreover, 
$$
\{\phi^S_\sigma\mid \sigma\in\nor_{\Sym_n}(G)\}
$$
is a set of $|\nor_{\Sym_n}(G)|$ distinct endomorphisms of $S$. 
In fact, let $\sigma,\sigma'\in\nor_{\Sym_n}(G)$ be such that $\sigma\neq\sigma'$. 
Then, by (\ref{I1}), $\phi_{I_1^{\POI_n}}^\sigma\neq\phi_{I_1^{\POI_n}}^{\sigma'}$ and so there exist $1\leqslant i,j\leqslant n$ 
such that $\binom{i}{j}\phi_{I_1^{\POI_n}}^\sigma\neq\binom{i}{j}\phi_{I_1^{\POI_n}}^{\sigma'}$. 
Hence, being $s\in J_{n-1}^S$ such that $\dom(s)=\Omega_n\setminus\{i\}$ and 
$\im(s)=\Omega_n\setminus\{j\}$, we have 
$$\textstyle 
s\phi^S_\sigma=s\phi_{\scriptscriptstyle 0}|_S\phi_{{\I^G_1}}^{\sigma}=\binom{i}{j}\phi_{{\I^G_1}}^{\sigma}=\binom{i}{j}\phi_{I_1^{\POI_n}}^{\sigma}\neq 
\binom{i}{j}\phi_{I_1^{\POI_n}}^{\sigma'}=\binom{i}{j}\phi_{{\I^G_1}}^{\sigma'}=s\phi_{\scriptscriptstyle 0}|_S\phi_{{\I^G_1}}^{\sigma'}=s\phi^S_{\sigma'}
$$
and so $\phi^S_{\sigma}\neq\phi^S_{\sigma'}$. 

Moreover, we have: 

\begin{lemma}\label{kerneln-1}
Let $S\in\{\POPI_n,\POP_n,\PORI_n,\POR_n\}$ and let 
$\phi$ be an endomorphism of $S$ such that $J_n^S\phi=J^S_n$, $J_{n-1}^{S}\phi\subseteq J_{1}^{S}$ and $I_{n-2}^{S}\phi=\{\emptyset\}$.
Then $\phi=\phi^S_\sigma$ for some $\sigma\in\nor_{\Sym_n}(J_n^S)$. 
\end{lemma} 
\begin{proof} 
First, observe that $\ker(\phi)=\rho_{\pi}^{n-1}$, where $\pi$ is the universal congruence on a maximal subgroup of $S$ contained in $J_{n-1}^S$. 
Hence, for $s,t\in J_{n-1}^S$, we have $s\phi=t\phi$ if and only if $s\H t$. 

\smallskip 

In order to simplify the notation, let us denote $J_n^S$ by $G$.  

Let $T\in\{\POI_n,\PO_n\}$ be such that $S\in\{\langle T\cup\{g\}\rangle,\langle T\cup\{g,h\}\rangle\}$. As $G\phi=G$ and $G$ is a subgroup of $S$, we get $1\phi=1$. Moreover, $I_{n-1}^S\phi\subseteq J_1^S\cup\{\emptyset\}=J_1^T\cup\{\emptyset\}$, whence $\phi|_T$ may be considered as an endomorphism of $T$. Furthermore, we have $1\phi|_T=1$, $J_{n-1}^T\phi|_T\subseteq J_1^T$ and $I_{n-2}^{T}\phi|_T=\{\emptyset\}$. Therefore, 
by \cite[Lemma 2.7]{Fernandes&Santos:2019}, we have 
$$
\phi|_T=\phi_{\scriptscriptstyle 0}|_T\phi_{I_1^1}^\sigma,
$$
for some $\sigma\in\Sym_n$. 

Let $1\leqslant i,j\leqslant n$ and take $s\in J_{n-1}^{\POI_n}$ 
such that $\dom(s)=\Omega_n\setminus\{i\}$ and $\im(s)=\Omega_n\setminus\{j\}$. 
Then $s\in J_{n-1}^T$ and $s\phi_{\scriptscriptstyle 0}=\binom{i}{j}$. 
Let $u\in G$. Then $us\in J_{n-1}^S$. 
Moreover, we also have $us\in J_{n-1}^{\PORI_n}$ and so 
there exists $t\in T$ such that $t\H us$. 
Since $\dom(us)=\Omega_n\setminus\{iu^{-1}\}$ and $\im(us)=\im(s)$, we get 
$$\textstyle 
t\phi_{\scriptscriptstyle 0}|_T =t\phi_{\scriptscriptstyle 0}=(us)\phi_{\scriptscriptstyle 0}=\binom{iu^{-1}}{j}. 
$$ 
Hence
\begin{align*}\textstyle  
\binom{iu^{-1}\sigma}{j\sigma} =\sigma^{-1}\binom{iu^{-1}}{j}\sigma=\binom{iu^{-1}}{j}\phi_{I_1^1}^\sigma
=t\phi_{\scriptscriptstyle 0}|_T \phi_{I_1^1}^\sigma = t\phi|_T=t\phi=(us)\phi = u\phi s\phi =  u\phi s\phi|_T = \\\textstyle 
= u\phi\cdot s\phi_{\scriptscriptstyle 0}|_T \phi_{I_1^1}^\sigma 
= u\phi\cdot \binom{i}{j} \phi_{I_1^1}^\sigma = u\phi\cdot \sigma^{-1}\binom{i}{j}\sigma 
= u\phi\cdot \binom{i\sigma}{j\sigma} = \binom{i\sigma(u\phi)^{-1}}{j\sigma}, 
\end{align*}
and so $iu^{-1}\sigma =i\sigma(u\phi)^{-1}$, i.e. $iu^{-1}\sigma(u\phi)=i\sigma$. 
Thus $u^{-1}\sigma(u\phi)=\sigma$ and so $u\phi=\sigma^{-1}u\sigma$. 
As $G\phi=G$, we may conclude that $\sigma^{-1}G\sigma\subseteq G$, i.e. $\sigma\in\nor_{\Sym_n}(G)$. 

\smallskip 

Now, take $s\in S$. 

If $s\in G$ then, by the above calculations, we have 
$s\phi=\sigma^{-1}s\sigma=s\phi^{\sigma}_{\I_1^G}=(s)\id\phi^{\sigma}_{\I_1^G} = s\phi_{\scriptscriptstyle 0}|_S\phi^{\sigma}_{\I_1^G}=s\phi^S_\sigma$. 

If $s\in J_{n-1}^S$ then $s=ut$, for some $t\in T$ and $u\in G$. Hence  
$$
t\phi=t\phi|_T = t\phi_{\scriptscriptstyle 0}|_T\phi_{I_1^1}^\sigma = (t\phi_{\scriptscriptstyle 0})\phi_{I_1^1}^\sigma = 
(t\phi_{\scriptscriptstyle 0})\phi_{\I_1^G}^\sigma 
= (t\phi_{\scriptscriptstyle 0}|_S)\phi_{\I_1^G}^\sigma 
= t\phi_{\scriptscriptstyle 0}|_S\phi_{\I_1^G}^\sigma
=t\phi^S_\sigma 
$$
and so 
$$
s\phi=(ut)\phi=u\phi t\phi = u\phi^S_\sigma t\phi^S_\sigma = (ut)\phi^S_\sigma = s\phi^S_\sigma. 
$$

Finally, if $s\in I_{n-2}^S$ then $s\phi=\emptyset=\sigma^{-1}\emptyset\sigma=\emptyset \phi^{\sigma}_{\I_1^G}
=(s\phi_{\scriptscriptstyle 0}|_S)\phi^{\sigma}_{\I_1^G}=s\phi^S_\sigma$. 

Thus $\phi=\phi^S_\sigma$ with $\sigma\in\nor_{\Sym_n}(J_n^S)$, as required. 
\end{proof}

Next, for each natural $m$, let $\ov{m}\in\Omega_n$ be such that $\ov{m}\equiv m\, \text{mod}\, n$ and, 
for $x\in\Omega_n$ and $k\in\Omega_n$ such that $\gcd(k,n)=1$, define a transformation $\sigma_{x,k}\in\T_n$ by 
$i\sigma_{x,k}=\ov{x+(i-1)k}$ for $i\in\Omega$, i.e. 
$$
\sigma_{x,k}= \begin{pmatrix}1 &2 &3 &\cdots &n\vspace*{.2em}\\
\ov{x} &\ov{x+k} &\ov{x+2k} &\cdots &\ov{x+(n-1)k}\end{pmatrix}. 
$$
From $\gcd(k,n)=1$, it is easy to deduce that $\sigma_{x,k}\in\Sym_n$. 
Let
$$
\mathcal{N}_n=\left\{\sigma_{x,k}\mid \mbox{$x,k\in\Omega_n$ and $\gcd(k,n)=1$}\right\}\subseteq\Sym_n.
$$

Let $\varphi:\N\longrightarrow\N$ be the Euler's totient function, i.e., 
for each $m\in\N$, $\varphi(m)$ is the number of naturals $k\in\{1,2,\ldots,m\}$ relatively prime to $m$. 
Clearly, $|\mathcal{N}_n|=n\varphi(n)$. 

\begin{lemma} \label{nn}
Let $n\geqslant3$ and $G\in\{\C_n,\D_{2n}\}$. Then $\nor_{\Sym_n}(G)=\mathcal{N}_n$ and so 
$|\nor_{\Sym_n}(G)|=n\varphi(n)$.
\end{lemma}
\begin{proof}
Let $\sigma\in \nor_{\Sym_n}(G)$. Then, we may consider the automorphism $\phi_G^\sigma$ of $G$. Since $n\geqslant3$, the only elements of $G$ of order $n$ are of the form $g^k$ with $\gcd(k,n)=1$. Then $g\phi_G^\sigma=g^k$, i.e. $\sigma^{-1}g\sigma=g^k$, for some $k\in\Omega_n$ such that $\gcd(k,n)=1$. 
Hence $g\sigma=\sigma g^{k}$ and so $ig\sigma=i\sigma g^{k}$, i.e. $(\ov{i+1})\sigma=\ov{i\sigma +k}$, for $i\in\Omega_n$. 
Let $x=1\sigma$. Then $i\sigma=\ov{(i-1)\sigma +k}=\ov{x + (i-1)k}$, for $2\leqslant i\leqslant n$. Thus $\sigma=\sigma_{x,k}\in \mathcal{N}_n$.

\smallskip 

Conversely, take $x\in\Omega_n$ and $k\in\Omega_n$ such that $\gcd(k,n)=1$. 
Let $\sigma=\sigma_{x,k}$. 
Then, it is a routine matter to show that 
$\sigma^{-1}g\sigma=g^k$ and $\sigma^{-1}h\sigma=hg^{\ov{2x-k-1}}$. 
Hence $\sigma^{-1}g^j\sigma=g^{\ov{kj}}$ and $\sigma^{-1}hg^j\sigma=hg^{\ov{2x+(j-1)k-1}}$ for $0\leqslant j\leqslant n-1$, 
from which follows that $\sigma\in \nor_{\Sym_n}(G)$, as required. 
\end{proof}

\medskip 

Next, we recall that Ara\'ujo et. al. proved in \cite[Theorem 5.5]{Araujo&etal:2011} that 
$$
\auto(S)=\langle \phi^g_S,\phi^h_S\rangle \simeq\D_{2n},
$$
for  $S\in\{\OP_n,\POPI_n,\POP_n\}$ and $n\geqslant3$. 
On the other hand, in the same paper, but only for $n\geqslant17$, it was also proved that 
$$
\auto(S)=\langle \phi^g_S,\phi^h_S\rangle \simeq\D_{2n},
$$
for  $S\in\{\OR_n,\PORI_n,\POR_n\}$ (see \cite[Theorem 5.6]{Araujo&etal:2011}). In what follows, 
by using a different approach, 
we aim to extend this last result to $n\geqslant3$. 
In fact, the result is trivial for $n=3$, by the already mentioned Sullivan's Theorem \cite[Theorem 2]{Sullivan:1975} 
(see also \cite[Corollary 4.1]{Araujo&etal:2011}), 
since $\OR_3=\T_3$, $\PORI_3=\I_3$, $\POR_3=\PT_3$ and $\Sym_3=\D_{2\cdot3}$. 

\smallskip 

Let $n\geqslant4$ and $S\in\{\OR_n,\PORI_n,\POR_n\}$. Take $T\in\{\OP_n,\POPI_n,\POP_n\}$ such that $S=\langle T\cup\{h\}\rangle$. 

Recall that $\O_n$ and $\PO_n$ are idempotent-generated (see \cite{Gomes&Howie:1992}) semigroups. 
Therefore, $\OP_n=\langle \O_n\cup\{g\}\rangle=\langle E(\O_n)\cup\{g\}\rangle$ and 
$\POP_n=\langle \PO_n\cup\{g\}\rangle=\langle E(\PO_n)\cup\{g\}\rangle$. 
On the other hand, $\POPI_n=\langle g,s_1\rangle$, where 
$$
s_1=\begin{pmatrix} 
1&2&\cdots&n-2&n-1\\
1&2&\cdots&n-2&n
\end{pmatrix} 
$$
(see \cite{Fernandes:2000}). Moreover, being 
 $$
g_{n-1}=\begin{pmatrix} 
1&2&\cdots&n-2&n-1\\
2&3&\cdots&n-1&1
\end{pmatrix}, 
$$
which is a group-element of $\POPI_n$ of order $n-1$, we obtain $g_{n-1}g^{n-1}=s_1$ and so we also have $\POPI_n=\langle g,g_{n-1}\rangle$. 
 
Let $\phi\in\auto(S)$. Since $g$ is a group-element of $S$ of order $n$, $g\phi$ must be also a group-element of $S$ of order $n$. 
Hence $g\phi\in T$, since $S\setminus T$ only has group-elements of order $2$. Also, for $S=\PORI_n$, as $g_{n-1}$ is a group-element of $S$ of order $n-1$ and $n-1\geqslant3$, likewise we deduce that $g_{n-1}\phi\in T$. On the other hand, as $\phi$ applies idempotents in idempotents and $E(S)=E(T)$, we can conclude that $T\phi\subseteq T$ and so we may consider $\phi|_T$ as a automorphism of $T$. 
 
Let $\Phi:\auto(S)\longrightarrow\auto(T)$ be the mapping defined by $\phi\Phi=\phi|_T$, for all $\phi\in\auto(S)$. 
Clearly, $\Phi$ is a group homomorphism.

Let $\varphi\in\auto(T)$. Then, by \cite[Theorem 5.5]{Araujo&etal:2011}, $\varphi=\phi^\sigma_T$, for some $\sigma\in\D_{2n}$. 
Obviously, $\sigma^{-1}S\sigma\subseteq S$ and so we may consider $\phi^\sigma_S\in\auto(S)$ and, trivially, we have 
$\phi^\sigma_S\Phi=\phi^\sigma_S|_T=\phi^\sigma_T=\varphi$. 
Hence $\Phi$ is a surjective homomorphism. 

Next, we prove that $\Phi$ is injective. 
Let $\phi_1,\phi_2\in\auto(S)$ be such that $\phi_1\Phi=\phi_2\Phi$. 
Since $S$ contains constant idempotent transformations with arbitrary images, by \cite[Lemma 2.2]{Araujo&etal:2011}, we have 
$\auto(S)=\inn(S)$ and so $\phi_i=\phi^{\sigma_i}_S$ for some $\sigma_i\in\Sym_n$ such that $\sigma_i^{-1}S\sigma_i\subseteq S$, for $i=1,2$. 
Hence, since $\D_{2n}$ is the group of units of $S$, we also have $\sigma_i^{-1}\D_{2n}\sigma_i\subseteq \D_{2n}$ and so, by Lemma \ref{nn}, 
$\sigma_i=\sigma_{x_i,k_i}$ for some $x_i,k_i\in\Omega_n$ such that $\gcd(k_i,n)=1$, for $i=1,2$. 
Now, by the proof of Lemma \ref{nn}, we have 
$$
g^{k_1}=\sigma_1^{-1}g\sigma_1=g\phi^{\sigma_1}_S|_T=g\phi_1|_T=g\phi_2|_T=g\phi^{\sigma_2}_S|_T = \sigma_2^{-1}g\sigma_2 =g^{k_2}
$$
and so $k_1=k_2$, since $1\leqslant k_1,k_2\leqslant n$. On the other hand, let 
$$
\mbox{$
e= \begin{pmatrix} 
1&2&3&\cdots&n\\
2&2&3&\cdots&n
\end{pmatrix}
$,
if $S=\OP_n$, and 
$
e= \begin{pmatrix} 
2&3&\cdots&n\\
2&3&\cdots&n
\end{pmatrix}
$,
if $S\in\{\PORI_n,\POR_n\}$.}
$$
Then, it is easy to check that 
$$
\Omega_n\setminus\{x_1\}=\im(\sigma_1^{-1}e\sigma_1)=\im(e\phi^{\sigma_1}_S|_T)=
\im(e\phi_1|_T)=\im(e\phi_2|_T)=\im(e\phi^{\sigma_2}_S|_T) = \im(\sigma_2^{-1}e\sigma_2)= \Omega_n\setminus\{x_2\}
$$
and so $x_1=x_2$. Hence $\sigma_1=\sigma_2$ and thus $\phi_1=\phi_2$, as required. 
 
Therefore $\Phi$ is an isomorphism and so we have: 

\begin{theorem}\label{autoor}
Let $n\geqslant3$ and $S\in\{\OR_n,\PORI_n,\POR_n\}$. Then 
$
\auto(S)=\langle \phi^g_S,\phi^h_S\rangle \simeq\D_{2n}. 
$
\end{theorem}

\medskip 

Now, we can present our main theorem: 

\begin{theorem}\label{endomain}
For $n\geqslant3$, let $S\in\{\OP_n,\POPI_n,\POP_n,\OR_n,\PORI_n,\POR_n\}$ and $\phi: S\longrightarrow S$ be any mapping. 
Then $\phi$ is an endomorphism of the semigroup $S$ if and only if one of the following properties holds:
\begin{enumerate}
\item $\phi=\phi_S^\sigma$, for some $\sigma\in\D_{2n}\simeq\auto(S)$, and so $\phi$ is an automorphism; 

\item $S\in\{\POPI_n,\POP_n,\PORI_n,\POR_n\}$ and $\phi=\phi^S_\sigma$, for some $\sigma\in \mathcal{N}_n$; 

\item there exist idempotents $e,f\in S$ with $e\not=f$ and $ef=fe=f$ such
that $J_n^S\phi=\{e\}$ and $I_{n-1}^S\phi=\{f\}$;

\item $S\in\{\POPI_n,\POP_n\}$ and 
there exists a group-element $g_0$ in $S$ of order $p$, with $p$ a divisor of $n$ and $p>1$, 
such that $g^i\phi=g_0^{i+p}$, for $0\leqslant i\leqslant n-1$, and $I_{n-1}^S\phi=\{\emptyset\}$;

\item $S\in\{\PORI_n,\POR_n\}$ and 
there exist group-elements $h_0$ of order $2$ and $g_0$ of order $p$, with $p$ a divisor of $n$ and $p>1$, in some maximal subgroup of $S$ such that 
$\langle g_0,h_0\rangle$ is a dihedral group of order $2p$, 
$g^i\phi=g_0^{i+p}$ and $(hg^i)\phi=h_0g_0^{i+p}$,  for $0\leqslant i\leqslant n-1$, 
and $I_{n-1}^S\phi=\{\emptyset\}$; 

\item $S\in\{\OR_n,\PORI_n,\POR_n\}$ and  
there exist a group-element $h_0$ in $S$ of order $2$ and an idempotent $f$ of $S$ with rank less than or equal to $2$ such that 
$h_0f=fh_0=f$, $I_{n-1}^S\phi=\{f\}$ and:
\begin{enumerate}
\item $g^i\phi=h_0^2$ and $(hg^i)\phi=h_0$, for $0\leqslant i\leqslant n-1$; or 
\item $g^i\phi=(hg^i)\phi=h_0^{i+2}$, for $0\leqslant i\leqslant n-1$, with $n$ even; or 
\item $g^i\phi=h_0^{i+2}$ and $(hg^i)\phi=h_0^{i+1}$, for $0\leqslant i\leqslant n-1$, with $n$ even. 
\end{enumerate}

\item $\phi$ is a constant mapping with idempotent value.
\end{enumerate}
\end{theorem} 

\begin{proof}
If $\phi: S\longrightarrow S$ is a mapping that checks 1, 2, 3 or 7 then it is clear that $\phi$ is an endomorphism of $S$. 
On the other hand, we can routinely verify that the mappings defined in 4, 5 and 6 are endomorphisms of $S$. 

\smallskip 

Conversely, let us admit that $\phi$ is an endomorphism of $S$. 

If $\ker(\phi)$ is the universal congruence then, obviously, Property 7 holds. 

Next, suppose that $\ker(\phi)=\rho^k_\pi$, where $\pi$ is a congruence on a fixed maximal subgroup contained in $J^S_k$, 
for some $k\in\{1,2,\ldots,n\}$. 

If $k=1$ then $\phi$ is an automorphism of $S$ and so, by \cite[Theorem 5.5]{Araujo&etal:2011} and Theorem \ref{autoor}, Property 1 holds. 

If $2\leq k\leq n-1$ then, by Lemmas \ref{2leqn-1}, \ref{kerneln-1} and \ref{nn}, Property 2 holds. 

Finally, if $k=n$ then, by Lemma \ref{rhon}, Property 3, 4, 5 or 6 holds, as required. 
\end{proof}

Let $S\in\{\OP_n,\POPI_n,\POP_n,\OR_n,\PORI_n,\POR_n\}$ and let $\phi: S\longrightarrow S$ be an endomorphism of $S$. 
For $i\in\{1,2,\ldots,7\}$, 
we say that $\phi$ is an endomorphism of \textit{type} $i$ if $\phi$ satisfies $i$ of the previous theorem. 

\section{On the number of endomorphisms} \label{endonumber}

In this section we will give formulae for the cardinals of the monoids of endomorphisms of the semigroups 
$\OP_n$, $\POPI_n$, $\POP_n$, $\OR_n$, $\PORI_n$ and $\POR_n$. 

\smallskip 

In what follows it is convenient to use the following notation. For a nonempty subset $X$ of $\Omega_n$, 
let $\O(X)=\T(X)\cap\PO_n$, $\OP(X)=\T(X)\cap\POP_n$ and $\POP(X)=\PT(X)\cap\POP_n$. 

\smallskip 

Let $S\in\{\OP_n,\POPI_n,\POP_n,\OR_n,\PORI_n,\POR_n \}$. For $e\in E(S)$, define  
$$
E_{S}(e)=\{f\in E(S)\mid ef=fe=f\}.
$$
Clearly, for $e\in E(S)$, $|E_S(e)|$ is the number of endomorphisms of $S$ of types $3$ and $7$ such that $J_n^S\phi=\{e\}$. 
Thus $\sum_{e\in E(S)}|E_S(e)|$ is the total number of endomorphisms of $S$ of types $3$ and $7$. 

\medskip 

As usual, for $i\geqslant0$, denote the $i$th Fibonacci number by $F_i$.  
Recall that $|E(\OP_n)|=F_{2n-1}+F_{2n+1}-n^2+n-2$ \cite[Theorem 2.10]{Catarino&Higgins:1999} and 
$|E(\POP_n)|=1+\sum_{k=1}^{n}\binom{n}{k}(\tau^{2k}+\theta^{2k}-k^{2}+k-2)$ \cite[Theorem 3.2]{Fernandes&Gomes&Jesus:2011}, 
where $\tau=\frac{1+\sqrt{5}}{2}$ and $\theta=\frac{1-\sqrt{5}}{2}$. 

\medskip 

Let $e$ be an idempotent of $S$ with rank $k\in\{1,2,\ldots,n\}$.  

\smallskip 

Suppose that $S\in\{\OP_n,\POP_n\}$ and let $V=\OP(\im(e))$, if $S=\OP_n$, or $V=\POP(\im(e))$, if $S=\POP_n$. 

Let $f\in E(V)$. 
Then, clearly, $ef\in S$ and, since $\im(f)\subseteq\im(e)=\fix(e)$, we have $fe=f$, 
whence $(ef)^2=ef$ and $(ef)e=ef=e(ef)$, i.e. $ef\in E_S(e)$. 
Thus, we may define a mapping $\psi: E(V)\longrightarrow E_S(e)$ by $f\psi=ef$ for all $f\in E(V)$. 

Let $u,v\in E(V)$ be such that $eu=ev$. Let $x\in\dom(u)$. Then $x\in\im(e)=\fix(e)$ and so $x=xe$, whence 
$x\in\dom(eu)=\dom(ev)$ and so $x\in\dom(v)$. Moreover, $xu=xeu=xev=xv$ and so $u=v$. 
Thus $\psi$ is injective. 

Let $f'\in E_S(e)$. Since $f'e=f'$, we have $\im(f')\subseteq\im(e)$ and so $f'|_{\im(e)}\in E(V)$. 
On the other hand, from $f'=ef'$ it follows that $\dom(f')=\dom(ef')$. 
Hence, if $x\in\dom(f')$ then $x\in\dom(e)$ and $xe\in\dom(f')\cap\im(e)$, i.e. $x\in\dom(e\cdot f'|_{\im(e)})$. 
Conversely, it is clear that $\dom(e\cdot f'|_{\im(e})\subseteq \dom(ef')=\dom(f')$, whence 
$\dom(e\cdot f'|_{\im(e)}) = \dom(f')$. Moreover, for  $x\in\dom(f')$, 
we get $x(e\cdot f'|_{\im(e)})=xef'=xf'$ and so 
$f'|_{\im(e)}\psi=e\cdot f'|_{\im(e)}=f'$. 
So $\psi$ is also surjective. 

Therefore $|E_{\OP_n}(e)|=|E(\OP(\im(e)))|=|E(\OP_k)|$ and $|E_{\POP_n}(e)|=|E(\POP(\im(e)))|=|E(\POP_k)|$. 

\smallskip 

Now, consider $S=\POPI_n$. Then, it is easy to check that $f\in E_{\POPI_n}(e)$ if and only if $\im(f)\subseteq\im(e)$, for any idempotent $f$ of $\POPI_n$.   
Hence  $|E_{\POPI_n}(e)|=2^{|\im(e)|}=2^k$. 

\smallskip 

Since $E(\OR_n)=E(\OP_n)$, $E(\PORI_n)=E(\POPI_n)$ and $E(\POR_n)=E(\POP_n)$, we proved: 

\begin{lemma}
Let $S\in\{\OP_n,\POPI_n,\POP_n,\OR_n,\PORI_n,\POR_n \}$ and let $e$ be an idempotent of $S$ with rank $k\in\{1,2,\ldots,n\}$.  
Then: 
\begin{enumerate}
\item $|E_S(e)|= F_{2k-1}+F_{2k+1}-k^2+k-2$, 
for $S\in\{\OP_n,\OR_n\}$; 
\item $|E_S(e)|= 2^k$, 
for $S\in\{\POPI_n,\PORI_n\}$;
\item $|E_S(e)|= 1+\sum_{i=1}^{k}\binom{k}{i}(\tau^{2i}+\theta^{2i}-i^{2}+i-2)$, 
for $S\in\{\POP_n,\POR_n\}$. 
\end{enumerate}
\end{lemma} 

\medskip 

Next, notice that $|E_S(e)|$ depends only on the rank of $e$, for all $e\in E(S)$. 
Therefore, the number of endomorphisms of $S$ of types 3 and 7 is: 
\begin{enumerate}
\item 
$\sum_{e\in E(S)}|E_S(e)|
=\sum_{k=1}^n|E(J_k^S)|(F_{2k-1}+F_{2k+1}-k^2+k-2)$,
for $S\in\{\OP_n,\OR_n\}$;
\item 
$\sum_{e\in E(S)}|E_S(e)|
=\sum_{k=0}^n|E(J_k^S)|2^k=\sum_{k=0}^n\binom{n}{k}2^k=3^n$,
for $S\in\{\POPI_n,\PORI_n\}$;
\item 
$\sum_{e\in E(S)}|E_S(e)|
=\sum_{k=0}^n|E(J_k^S)|(1+\sum_{i=1}^{k}\binom{k}{i}(\tau^{2i}+\theta^{2i}-i^{2}+i-2))$,
for $S\in\{\POP_n,\POR_n\}$. 
\end{enumerate}

\smallskip 

Notice that it is clear that $|E(J_k^{\POPI_n})|=\binom{n}{k}$, for $k\in\{0,1,\ldots,n\}$. 

\medskip

Next, we compute the number of idempotents of $\OP_n$ with rank $k$ for $k\in\{1,2,\ldots,n\}$ . 

\begin{lemma}\label{EJOPnk}
Let $k\in\{1,2,\ldots,n\}$. Then the number of idempotents of $\OP_n$ with rank $k$ is 
\begin{equation*}
\begin{cases}
n \quad & \mbox{if $k=1$}\\
\sum_{i=1}^{n-k+1}i^{2}\binom{n-i+k-2}{2k-3} \quad &\mbox{if $k\geqslant2$}.
\end{cases}
\end{equation*}
\end{lemma}
\begin{proof}
If $k=1$ then, clearly, $|E(J^{\OP_n}_{1})|=n$.

So,  consider $k\geqslant2$. Let $e\in E(J^{\OP_n}_{k})$ and $i=|(1e)e^{-1}|$. 
Then $1\leqslant i\leqslant n-k+1$ and we may build $e$ as described below. 

First, there are $i$ choices for the convex set $(1e)e^{-1}$, namely $\{1,2,\ldots,i\}$, $\{1,\ldots,i-1,n\}$,  $\{1,\ldots,i-2,n-1,n\}$, \ldots, 
$\{1,2,n-i+3,\ldots,n\}$ and $\{1,n-i+2,\ldots,n\}$. Once $(1e)e^{-1}$ is chosen, there are $i$ choices for $1e$. 

Now, let $M=(1e)e^{-1}$.  
Then $|\Omega_n\setminus M|=n-i$ and $e|_{\Omega_n\setminus M}$ is an idempotent of $\O(\Omega_n\setminus M)$ with rank $k-1$. 
Hence, by \cite[Corollary 4.4]{Laradji&Umar:2006}, we $e$ can be extended from $e|_{M}$ to an idempotent of rank $k$ of $\OP_n$ in 
$|E(J^{\O(\Omega_n\setminus M)}_{k-1})|=\binom{n-i+k-2}{2k-3}$ ways and so 
the number of such idempotents $e$ is $i^{2}\binom{n-i+k-2}{2k-3}$. 

Thus, since $1\leqslant i\leqslant n-k+1$, we conclude that $E(J^{\OP_n}_{k})=\sum_{i=1}^{n-k+1}i^{2}\binom{n-i+k-2}{2k-3}$, as required.
\end{proof}

\smallskip 

At this point, we can count the number of endomorphisms of $\OP_n$. 
Since there are $2n$ endomorphisms of $\OP_n$ of type 1 (i.e. automorphisms) and the remaining endomorphisms of $\OP_n$ are of types 3 and 7, 
we have:  

\begin{theorem}
For $n\geqslant3$, the semigroup $\OP_n$ has 
$$
3n + \sum_{k=2}^{n}\sum_{i=1}^{n-k+1}i^{2}\binom{n-i+k-2}{2k-3}(F_{2k-1}+F_{2k+1}-k^2+k-2) 
$$
endomorphisms. 
\end{theorem} 

\medskip 

Now, we compute the number of idempotents of $\POP_n$ with rank $k$ for $k\in\{1,2,\ldots,n\}$. 
Of course we have $|E(J_0^{\POP_n})|=1$.  

\begin{lemma} \label{EJPOPnk}
Let $k\in\{1,2,\ldots,n\}$. Then the number of idempotents of $\POP_n$ with rank $k$ is 
\begin{equation*}
\begin{cases}
n2^{n-1}\quad & \mbox{if $k=1$}\\
\sum_{i=k}^{n}(\binom{n}{i}\sum_{j=1}^{i-k+1}j^{2}\binom{i-j+k-2}{2k-3}) \quad &\mbox{if $k\geqslant2$}.
\end{cases}
\end{equation*}
\end{lemma}
\begin{proof}
Given an idempotent $e$ of $\POP_n$, we have 
$\im(e)=\fix(e)\subseteq\dom(e)$ and so we may consider $e\in\OP(\dom(e))$. Hence, an element 
$e$ of $\POP_n$ is an idempotent with rank $k$ if and only if $e$ is an idempotent of $\OP(X)$ with rank $k$, for some subset $X$ of $\Omega_n$ such that $|X|\ge k$. 
Thus, the number of idempotents of $\POP_n$ with rank $k$ is 
$
\sum_{i=k}^{n}\binom{n}{i}|E(J_k^{\OP_i})|, 
$
and so the result follows, by Lemma \ref{EJOPnk} and noticing that $\sum_{i=1}^{n}i\binom{n}{i} =n2^{n-1}$. 
\end{proof}

\smallskip 

The previous lemma allows us to express the number of endomorphisms of $\POP_n$ of types $3$ and $7$ by 
$$
\begin{array}{cl}
& \sum_{k=0}^n|E(J_k^{\POP_n})|(1+\sum_{i=1}^{k}\binom{k}{i}(\tau^{2i}+\theta^{2i}-i^{2}+i-2)) \vspace*{.3em}\\
= & \sum_{k=0}^n|E(J_k^{\POP_n})| + \sum_{k=0}^n|E(J_k^{\POP_n})| \sum_{i=1}^{k}\binom{k}{i}(\tau^{2i}+\theta^{2i}-i^{2}+i-2) \vspace*{.3em}\\
= & |E(\POP_n)| + |E(J_1^{\POP_n})|  + \sum_{k=2}^n|E(J_k^{\POP_n})| \sum_{i=1}^{k}\binom{k}{i}(\tau^{2i}+\theta^{2i}-i^{2}+i-2) \vspace*{.3em}\\
= & 1+\sum_{k=1}^{n}\binom{n}{k}(\tau^{2k}+\theta^{2k}-k^{2}+k-2) + n2^{n-1} + \vspace*{.3em}\\
& \qquad\qquad+\sum_{k=2}^n (\sum_{i=k}^{n}(\binom{n}{i}\sum_{j=1}^{i-k+1}j^{2}\binom{i-j+k-2}{2k-3}))(\sum_{i=1}^{k}\binom{k}{i}(\tau^{2i}+\theta^{2i}-i^{2}+i-2))
\end{array}
$$ 

\medskip 

Let us now count the number of endomorphisms of $S\in\{\POPI_n,\POP_n\}$ of type $4$. 
Since the number of maximal subgroups of a semigroup is the number of idempotents of the semigroup and maximal subgroups contained in the same 
$\mathscr{J}$-class are isomorphic, the number of endomorphisms of $S$ of type $4$ is 
$$
\sum_{k=2}^n|E(J_k^S)|\varepsilon(n,k), 
$$ 
where $\varepsilon(n,k)$ is the number of elements of a cyclic group of order $k$ (i.e. the maximal subgroups of $J_k^S$) 
which have as order a divisor of $n$ greater than one, i.e. 
$$
\varepsilon(n,k)=|\{t\in\{1,2,\ldots,k-1\}\mid \mbox{$\frac{k}{\gcd(t,k)}$ divides $n$}\}|. 
$$

Let $S\in\{\POPI_n,\POP_n\}$. 
Since there are $2n$ endomorphisms of $S$ of type $1$,  
$n\varphi(n)$ endomorphisms of $S$ of type $2$, 
and the remaining endomorphisms of $S$ are of types $3$, $4$ and $7$, 
we have:  

\begin{theorem}
For $n\geqslant3$, the semigroups $\POPI_n$ and $\POP_n$ have 
$$
2n+n\varphi(n)+3^n+ \sum_{k=2}^n\binom{n}{k}\varepsilon(n,k), 
$$
and 
$$
\begin{array}{l}\displaystyle 
1+2n+n\varphi(n) + n2^{n-1}  + \sum_{k=1}^{n}\binom{n}{k}(\tau^{2k}+\theta^{2k}-k^{2}+k-2)\, + \\ \displaystyle 
\qquad\qquad
+\sum_{k=2}^n    
\left( 
\sum_{i=k}^{n}\binom{n}{i}\sum_{j=1}^{i-k+1}j^{2}\binom{i-j+k-2}{2k-3}\left(\varepsilon(n,k)+\sum_{i=1}^{k}\binom{k}{i}(\tau^{2i}+\theta^{2i}-i^{2}+i-2)\right) 
\right)
\end{array}
$$
endomorphisms, respectively. 
\end{theorem} 

\medskip 

Next, for $S\in\{\PORI_n,\POR_n\}$, we count the number of endomorphisms of $S$ of type 5. First, we recall some properties of dihedral groups. 

\smallskip 

Let $k\geqslant3$ and let 
$
\D_{2k}=\langle x,y\mid x^k=y^2=1, xy=yx^{k-1}\rangle=\{1,x,x^2,\ldots,x^{k-1}, y,yx,yx^2,\ldots,yx^{k-1}\} 
$
be a dihedral group of order $2k$. Then, it is a routine matter to check that the 
subgroup $\langle x^i,yx^j\rangle$ of $\D_{2k}$ is a dihedral group of order $2\frac{k}{\gcd(i,k)}$ 
(observe that $x^i$ and $yx^j$ have orders $\frac{k}{\gcd(i,k)}$ and $2$, respectively), 
for all $1\leqslant i\leqslant k-1$ and $0\leqslant j\leqslant k-1$. 
On the other hand, for all $0\leqslant i,j\leqslant k-1$, we have $\langle yx^i,yx^j\rangle=\langle x^{|i-j|},yx^j\rangle$ and so  
 $\langle yx^i,yx^j\rangle$ is a dihedral group of order $2\frac{k}{\gcd(|i-j|,k)}$, if $i\neq j$, and a cyclic group of order $2$, if $i=j$.  
 In particular, for an even $k$, $\langle yx^i,yx^j\rangle$ is a dihedral group of order $2\cdot 2$ if and only if $|i-j|=\frac{k}{2}$. 

\smallskip 

Let $k\geqslant3$ and let $H_k$ be a group $\mathscr{H}$-class of $S$ contained in $J_k^S$. Then $S$ admits $k\mspace{1.5mu}\varepsilon(n,k)$ endomorphisms 
$\phi$ of type $5$ such that $D_{2n}\phi\subseteq H_k$, if $n$ is odd or $k$ is odd, and $k\mspace{1.5mu}\varepsilon(n,k)+2k$ endomorphisms 
$\phi$ of type $5$ such that $D_{2n}\phi\subseteq H_k$, if $n$ and $k$ are both even. 
Thus, the number of endomorphisms of $S$ of type 5 is 
$$
\sum_{k=3}^n|E(J_k^S)|k(\varepsilon(n,k)+2\delta(n,k)), 
$$ 
where $\delta(n,k)=\frac{(-1)^{n+k}+(-1)^n+(-1)^k+1}{4}=
\left\{\begin{array}{ll}
1 & \mbox{if $n$ and $k$ are even}\\
0 & \mbox{otherwise.}
\end{array}\right.$

\medskip 

Finally, for $S\in\{\OR_n,\PORI_n,\POR_n\}$, we determine the number of endomorphisms of $S$ of type 6. 

\smallskip 

For any subset $X$ of $S$, denote by $G_X(2)$ be the set of all group-elements of order $2$ of $S$ that belong to $X$. 
For $h_0\in G_S(2)$, define 
$$
E_S(h_0)=\{f\in E(S)\mid h_0f=fh_0=f\}
$$
(notice that $fh_0=f$ if and only if $\fix(f)\subseteq\fix(h_0)$, for $h_0\in S$ and $f\in E(S)$). 
Then, the number of endomorphisms of $S$ of type 6 is equal to  
$$
\sum_{h_0\in G_S(2)}((-1)^n+2)|E_S(h_0)|. 
$$

Recall that $|\fix(s)|\leqslant2$, for all non-idempotent group-element $s$ of $S$. 
Let us state this property with more details.  

Let $h_0$ be a non-idempotent group-element of $S$. Then $k=\rank(h_0)\geqslant2$. If $h_0$ preserves the orientation then $\fix(h_0)=\emptyset$. 
Next, suppose that $h_0$ does not preserve the orientation. Then $h_0\in G_S(2)$, $k\geqslant3$ and, if $\im(h_0)=\{i_1<i_2<\cdots<i_k\}$, 
we have 
\begin{equation}\label{h0}
h_0|_{\im(h_0)}=
\left(\begin{matrix}
i_1&\cdots &i_\ell & i_{\ell+1}&\cdots&i_k\\
i_\ell&\cdots &i_1 & i_k&\cdots&i_{\ell+1}
\end{matrix}\right)
\end{equation}
(i.e. $i_th_0=i_{\ell-t+1}$, for $1\leqslant t\leqslant \ell$, and $i_th_0=i_{k-t+\ell+1}$, for $\ell+1\leqslant t\leqslant k$),  
for some $0\leqslant \ell\leqslant k-1$, and so 
$$
|\fix(h_0)|=\left\{
\begin{array}{lll}
0 & \mbox{if $\ell$ and $k-\ell$ are both even} & \mbox{(in this case, $k$ is even)} \\
2 & \mbox{if $\ell$ and $k-\ell$ are both odd} & \mbox{(in this case, $k$ is even)} \\
1 & \mbox{otherwise} & \mbox{(in this case, $k$ is odd).} 
\end{array}
\right. 
$$
Take $f\in E_S(h_0)$. Then, in particular, $h_0f=f$. 
So, for $x\in\Omega_n$, we have $x\in\dom(f)$ if and only if $xh_0\in\dom(f)$ and, in this case, $xf=(xh_0)f$. 
Hence, for $1\leqslant t\leqslant \ell$ such that $i_t\in\dom(f)$, 
we have $i_{\ell-t+1}\in\dom(f)$ and $i_{\ell-t+1}f=i_tf$ and, 
similarly,  for $\ell+1\leqslant t\leqslant k$ such that $i_t\in\dom(f)$, 
we have $i_{k-t+\ell+1}\in\dom(f)$ and $i_{k-t+\ell+1}f=i_tf$. 
Thus $f$ is perfectly defined by its images of $\dom(f)\cap\{i_{\lceil\frac{\ell+1}{2}\rceil},\ldots,i_\ell,i_{\ell+1},\ldots,i_{\ell+\lceil\frac{k-\ell}{2}\rceil}\}$ and 
noticing that 
$$
\dom(f)=\bigcup\left\{ xh_0^{-1}\cup (xh_0)h_0^{-1}\mid x \in \dom(f)\cap\{i_{\lceil\frac{\ell+1}{2}\rceil},\ldots,i_\ell,i_{\ell+1},\ldots,i_{\ell+\lceil\frac{k-\ell}{2}\rceil}\} \right\}. 
$$

\smallskip 

Now, let $k\geqslant2$ and $H_k$ be a group $\mathscr{H}$-class of $S$ contained in $J_k^S$. 

If $k=2$ then all elements of $H_k$ preserve the orientation and $H_k$ has a single element of order $2$ ($|H_k|=2$). 

So, suppose that $k\geqslant3$. Then $|H_k|=2k$ and $H_k$ has $\frac{(-1)^k+1}{2}$ elements of order $2$ preserving the orientation 
and $k$ elements of order $2$ that do not preserve the orientation.  
Therefore, if $k$ is odd then all of the $k$ elements of order $2$ of $H_k$ fix exactly one element. 
On the other hand, 
if $k$ is even then $\frac{k}{2}$ of the $k+1$ elements of order $2$ of $H_k$ fix exactly two elements while the remaining $\frac{k}{2}+1$ fix none.

\medskip 

Let $h_0$ be a group-element of $\PORI_n$ with order $2$. Then, clearly, 
$$
|E_{\PORI_n}(h_0)| =
\left\{
\begin{array}{ll}
1 & \mbox{if $|\fix(h_0)|=0$}\\
2 & \mbox{if $|\fix(h_0)|=1$}\\
4 & \mbox{if $|\fix(h_0)|=2$.}
\end{array}
\right.
$$

\smallskip 

Next, suppose that $h_0$ is a group-element of $\OR_n$ with order $2$. Then, clearly, 
$$
|E_{\OR_n}(h_0)| =
\left\{
\begin{array}{ll}
0 & \mbox{if $|\fix(h_0)|=0$}\\
1 & \mbox{if $|\fix(h_0)|=1$.}
\end{array}
\right.
$$
On the other hand, suppose that  $|\fix(h_0)|=2$. 
Then $k$ is even and, by considering $h_0$ as in (\ref{h0}), $\ell$ and $k-\ell$ must be both odd and $h_0$ fixes 
$i_\frac{\ell+1}{2}$ and $i_{\ell+\frac{k-\ell+1}{2}}$.
Since each element $f$ of $E_{\OR_n}(h_0)$ is perfectly defined by its images of 
$\{i_{\frac{\ell+1}{2}},\ldots,i_\ell,i_{\ell+1},\ldots,i_{\ell+\frac{k-\ell+1}{2}}\}$  
and $f$ preserves the orientation (since it is idempotent), 
we have $\frac{k}{2}$ such elements with rank $2$ 
(with image and fixed points $i_\frac{\ell+1}{2}$ and $i_{\ell+\frac{k-\ell+1}{2}}$) and $2$ elements with rank $1$ 
(one with image $i_\frac{\ell+1}{2}$ and the other with image $i_{\ell+\frac{k-\ell+1}{2}}$). Thus 
$$
|E_{\OR_n}(h_0)| =\frac{k}{2}+2.
$$ 

\smallskip 

Lastly, let $h_0$ be a group-element of $\POR_n$ with order $2$ under the conditions of (\ref{h0}). 

If $|\fix(h_0)|=0$ then $E_{\POR_n}(h_0)=\{\emptyset\}$ and so $|E_{\POR_n}(h_0)|=1$. 

On the other hand, for an element $f\in E_{\POR_n}(h_0)$, we have 
$i_t\in\dom(f)$ if and only if $i_{\ell-t+1}\in\dom(f)$, for $1\leqslant t\leqslant \ell$, 
and $i_t\in\dom(f)$ if and only if $i_{k-t+\ell+1}\in\dom(f)$, for $\ell+1\leqslant t\leqslant k$, 
which allows us to deduce: 
\begin{enumerate}
\item if $|\fix(h_0)|=1$ then $k$ is odd and there are 
$
\sum_{i=0}^{\frac{k-1}{2}}\binom{\frac{k-1}{2}}{i}
$
possible domains for an element $f\in E_{\POR_n}(h_0)$ with rank $1$; 

\item if $|\fix(h_0)|=2$ then $k$ is even and there are 
$
\sum_{i=0}^{\frac{k}{2}}\binom{\frac{k}{2}}{i}
$
possible domains for an element $f\in E_{\POR_n}(h_0)$ with rank $1$ 
and 
$
\sum_{i=0}^{\frac{k}{2}-1}\binom{\frac{k}{2}-1}{i}
$
possible domains for an element $f\in E_{\POR_n}(h_0)$ with rank $2$. 
\end{enumerate} 
Thus, we obtain 
$$
|E_{\POR_n}(h_0)| =
\left\{
\begin{array}{ll}
1+  \sum_{i=0}^{\frac{k-1}{2}}\binom{\frac{k-1}{2}}{i} & \mbox{if $|\fix(h_0)|=1$}\vspace*{.3em}\\
1+ 2\sum_{i=0}^{\frac{k}{2}}\binom{\frac{k}{2}}{i} + \sum_{i=0}^{\frac{k}{2}-1}\binom{\frac{k}{2}-1}{i}\frac{2+2i}{2} & \mbox{if $|\fix(h_0)|=2$,}
\end{array}
\right.
$$
i.e. 
$$
|E_{\POR_n}(h_0)| =
\left\{
\begin{array}{ll}
1+ 2^{\frac{k-1}{2}} & \mbox{if $|\fix(h_0)|=1$}\vspace*{.3em}\\
1+ (\frac{k}{2}+9)2^{\frac{k}{2}-2}  & \mbox{if $|\fix(h_0)|=2$.}
\end{array}
\right.
$$

\medskip 

Now, let $H_k$ be a fixed group $\mathscr{H}$-class of $S$ contained in $J_k^{S}$, for $2\leqslant k\leqslant n$. 
Then, since $|E_S(h_0)|$ depends only on the rank and the number of fixed points of $h_0$ for all $h_0\in G_S(2)$, 
we can rewrite the number of endomorphisms of $S$ of type $6$ as follows: 
$$
\begin{array}{rcl}
\sum_{h_0\in G_S(2)}((-1)^n+2)|E_S(h_0)| &= & 
((-1)^n+2)\sum_{h_0\in G_S(2)}|E_S(h_0)| \vspace*{.3em}\\ 
& = & 
((-1)^n+2)\sum_{k=2}^n|E(J_k^S)|\sum_{h_0\in G_{H_k}(2)}|E_S(h_0)|. 
\end{array}
$$

\medskip 

Let us consider the semigroup $\PORI_n$.  
Since $\sum_{h_0\in G_{H_2}(2)}|E_{\PORI_n}(h_0)|=1$ and, for $k\geq3$, 
$$
\sum_{h_0\in G_{H_k}(2)}|E_{\PORI_n}(h_0)|= 
\left\{
\begin{array}{ll}
2k & \mbox{if $k$ is odd}\\
(1+\frac{k}{2})+4\frac{k}{2} & \mbox{if $k$ is even} 
\end{array}
\right.
=~ 2k + \frac{(-1)^k+1}{2}\left(1+\frac{k}{2}\right), 
$$ 
then the number of endomorphisms of $\PORI_n$ of type 6 is equal to 
$$
((-1)^n+2)\left(
\binom{n}{2}+\sum_{k=3}^n\binom{n}{k}\left(2k + \frac{(-1)^k+1}{2}\left(1+\frac{k}{2}\right)\right)
\right), 
$$ 
i.e. 
$$
((-1)^n+2)\left((9n+4)2^{n-3}-\frac{5}{2}n(n-1)-2n-1\right). 
$$ 

\smallskip 

Next, consider the semigroup $\OR_n$. 
Since $\sum_{h_0\in G_{H_2}(2)}|E_{\OR_n}(h_0)|=0$ and, for $k\geq3$, 
$$
\sum_{h_0\in G_{H_k}(2)}|E_{\OR_n}(h_0)|= 
\left\{
\begin{array}{ll}
k & \mbox{if $k$ is odd}\\
\frac{k}{2}(\frac{k}{2}+2) & \mbox{if $k$ is even} 
\end{array}
\right.
=~ k + \frac{(-1)^k+1}{8} k^2, 
$$ 
then, by Lemma \ref{EJOPnk}, the number of endomorphisms of $\OR_n$ of type 6 is equal to 
$$
((-1)^n+2)
\sum_{k=3}^n  \sum_{i=1}^{n-k+1}i^{2}\binom{n-i+k-2}{2k-3}  \left(k + \frac{(-1)^k+1}{8} k^2\right).  
$$ 

\smallskip 

Finally, consider the semigroup $\POR_n$. 
Since $\sum_{h_0\in G_{H_2}(2)}|E_{\POR_n}(h_0)|=1$ and, for $k\geq3$, 
$$
\begin{array}{rcl} 
\sum_{h_0\in G_{H_k}(2)}|E_{\POR_n}(h_0)| & = &  
\left\{
\begin{array}{ll}
k (1+2^{\frac{k-1}{2}}) & \mbox{if $k$ is odd}\\
\frac{k}{2}(1+ (\frac{k}{2}+9)2^{\frac{k}{2}-2}) + \frac{k}{2} +1& \mbox{if $k$ is even} 
\end{array}
\right. \vspace*{.5em}\\
& = & k + k 2^{\lfloor\frac{k}{2}\rfloor}   +  \frac{(-1)^k+1}{2} (1 + \frac{k^2+2k}{16}2^{\lfloor\frac{k}{2}\rfloor}), 
\end{array} 
$$ 
then, by Lemma \ref{EJPOPnk}, the number of endomorphisms of $\POR_n$ of type 6 is equal to 
\begin{align*}
((-1)^n+2)\Bigg( \sum_{i=2}^{n}\binom{n}{i}\frac{i^2(i^2-1)}{12} + \hspace*{30em} \\   
 + \sum_{k=3}^{n} \sum_{i=k}^{n}
\binom{n}{i}\sum_{j=1}^{i-k+1}j^{2}\binom{i-j+k-2}{2k-3}
\bigg( k + k 2^{\lfloor\frac{k}{2}\rfloor}   +  \frac{(-1)^k+1}{2} \bigg(1 + \frac{k^2+2k}{16}2^{\lfloor\frac{k}{2}\rfloor}\bigg) \bigg) 
\Bigg) 
\end{align*} 

\medskip 

Now, we can count the total number of endomorphisms of $\OR_n$, $\PORI_n$ and $\POR_n$. 

\smallskip 

Since the number of endomorphisms of types 1, 3 and 7 of $\OR_n$ and $\OP_n$ coincide, we immediately have:  

\begin{theorem}
For $n\geqslant3$, the semigroup $\OR_n$ has 
\begin{align*}
3n + \sum_{k=2}^{n}\sum_{i=1}^{n-k+1}i^{2}\binom{n-i+k-2}{2k-3}(F_{2k-1}+F_{2k+1}-k^2+k-2) 
\,+  \hspace*{10em} \\ +\,  
((-1)^n+2)
\sum_{k=3}^n  \sum_{i=1}^{n-k+1}i^{2}\binom{n-i+k-2}{2k-3}  \bigg(k + \frac{(-1)^k+1}{8} k^2\bigg) 
\end{align*} 
endomorphisms. 
\end{theorem} 

Also the number of endomorphisms of types 1, 2, 3 and 7 of $\PORI_n$ and $\POPI_n$ [respectively, $\POR_n$ and $\POP_n$] coincide. 
Therefore, we get:  

\begin{theorem}
For $n\geqslant3$, the semigroups $\PORI_n$ and $\POR_n$ have 
$$
n\varphi(n)+3^n+ \sum_{k=3}^n\binom{n}{k}k(\varepsilon(n,k)+2\delta(n,k)) 
+ ((-1)^n+2)\big((9n+4)2^{n-3}-\frac{5}{2}n(n-1)-1\big) - 2((-1)^n+1)n, 
$$
and 
$$
\begin{array}{l}\displaystyle 
1+2n+n\varphi(n) + n2^{n-1}  + \sum_{k=1}^{n}\binom{n}{k}(\tau^{2k}+\theta^{2k}-k^{2}+k-2)\, 
+ \frac{(-1)^n+7}{12}\sum_{i=2}^{n}\binom{n}{i}i^2(i^2-1)
\, + \hspace*{0em}\\ \displaystyle 
\hspace*{2em} +\sum_{k=3}^n    
\sum_{i=k}^{n}\binom{n}{i}\sum_{j=1}^{i-k+1}j^{2}\binom{i-j+k-2}{2k-3}
\Bigg(
k(\varepsilon(n,k)+2\delta(n,k))+
\sum_{i=1}^{k}\binom{k}{i}(\tau^{2i}+\theta^{2i}-i^{2}+i-2)  \, +   \\ \displaystyle 
\hspace*{19em} + ((-1)^n+2) \bigg( k + k 2^{\lfloor\frac{k}{2}\rfloor}   
+ \,  \frac{(-1)^k+1}{2} \bigg(1 + \frac{k^2+2k}{16}2^{\lfloor\frac{k}{2}\rfloor}\bigg) \bigg) 
\Bigg)
\end{array}
$$
endomorphisms, respectively. 
\end{theorem} 

\section*{Appendix} 

In this paragraph, for completeness, we present a note on the endomorphisms of the groups $\C_n$ and $\D_{2n}$ ($n\geqslant3$). 

\smallskip

It is well known that $\C_n$ has $n$ endomorphisms among which $\varphi(n)$ are automorphisms. In fact, 
$$
\endo(\C_n)=\{\phi_i\mid 0\leqslant i\leqslant n-1\},
$$
where $\phi_i$ is the endomorphism of $\C_n$ induced by the mapping $g\mapsto g^i$ for $0\leqslant i\leqslant n-1$, 
and 
$$
\auto(\C_n)=\{\phi_i\mid \mbox{$1\leqslant i\leqslant n-1$ and $\gcd(i,n)=1$}\}. 
$$

\smallskip 

Now, let us consider the dihedral group $\D_{2n}$. 

First, consider the endomorphism $\phi_0$ of $\D_{2n}$ induced by the mapping $g\mapsto 1, h\mapsto 1$ and,  
for $0\leqslant i,j\leqslant n-1$, 
the endomorphism $\phi_{i,j}$ of $\D_{2n}$ induced by the mapping $g\mapsto g^i, h\mapsto hg^j$. 

If $n$ is an odd number then it is easy to show that 
$$
\endo(\D_{2n})=\{\phi_{i,j}\mid 0\leqslant i,j\leqslant n-1\}\cup\{\phi_0\} 
$$ 
and 
$$
\auto(\D_{2n})=\{\phi_{i,j}\mid \mbox{$0\leqslant i,j\leqslant n-1$ and $\gcd(i,n)=1$}\}. 
$$
Hence, for $n$ odd, $\D_{2n}$ has $n^2+1$ endomorphisms and $n\varphi(n)$ automorphisms. 

Next, suppose that $n$ is even and consider the following endomorphisms of $\D_{2n}$:
\begin{description}
\item\-- $\phi_n$ induced by the mapping $g\mapsto1, h\mapsto g^{\frac{n}{2}}$; 

\item\-- $\xi_{i,j}$ induced by the mapping $g\mapsto hg^i, h\mapsto hg^j$, for  $0\leqslant i,j\leqslant n-1$ such that $|i-j|={\frac{n}{2}}$; 

\item\-- $\xi_i$ induced by the mapping $g\mapsto hg^i, h\mapsto g^{\frac{n}{2}}$, for  $0\leqslant i\leqslant n-1$; 

\item\-- $\mu_i$ induced by the mapping $g\mapsto hg^i, h\mapsto1$, for  $0\leqslant i\leqslant n-1$; 

\item\-- $\mu_n$ induced by the mapping $g\mapsto g^{\frac{n}{2}}, h\mapsto1$; 

\item\-- $\nu_i$ induced by the mapping $g\mapsto hg^i, h\mapsto hg^i$, for  $0\leqslant i\leqslant n-1$; 

\item\-- $\nu_n$ induced by the mapping $g\mapsto g^{\frac{n}{2}}, h\mapsto g^{\frac{n}{2}}$.  
\end{description}
Then it is not difficult to prove that 
$$
\begin{array}{rcl}
\endo(\D_{2n})&=&\{\phi_{i,j}\mid 0\leqslant i,j\leqslant n-1\}
\cup\{\phi_0,\phi_n\} 
\cup\{\xi_{i,j}\mid \mbox{$0\leqslant i,j\leqslant n-1$ such that $|i-j|={\frac{n}{2}}$}\}\\ \\
&&\cup\, \{\xi_i\mid 0\leqslant i\leqslant n-1\} 
\cup\{\mu_i\mid 0\leqslant i\leqslant n\}
\cup\{\nu_i\mid 0\leqslant i\leqslant n\}. 
\end{array}
$$
and 
$$
\auto(\D_{2n})=\{\phi_{i,j}\mid \mbox{$0\leqslant i,j\leqslant n-1$ and $\gcd(i,n)=1$}\} 
$$
Thus, for $n$ even, $\D_{2n}$ has $n^2+4n+4$ endomorphisms and $n\varphi(n)$ automorphisms. 

\smallskip 

Notice that it is a routine matter to check that all mentioned endomorphisms of $\D_{2n}$ exist. 


 
\lastpage 
\end{document}